\documentclass{article}
\usepackage[utf8]{inputenc}
\usepackage[utf8]{inputenc}

\usepackage{times}
\usepackage[papersize={17cm,24cm},margin=22mm,top=17mm,headsep=5mm]{geometry}
\usepackage[dvipsnames]{xcolor}
\usepackage[english]{babel} 

\definecolor{bred}{rgb}{0.8,0,0}
\usepackage{dsfont}
\usepackage{microtype}
\usepackage{graphicx}
\usepackage{subfigure}
\usepackage{booktabs} 
\usepackage{amsfonts}
\usepackage{amsmath,amssymb,graphicx,epsfig}
\usepackage{enumerate,amsthm,epstopdf}

\usepackage[utf8]{inputenc} 
\usepackage[T1]{fontenc}    

\usepackage{bbm}
\usepackage{xargs}
\usepackage{enumerate}
\usepackage{latexsym}
\usepackage{wasysym}

\usepackage{float}
\usepackage{hyperref}       
\usepackage{url}            
\usepackage{booktabs}       
\usepackage{amsfonts}       
\usepackage{nicefrac}       
\usepackage{microtype}      
\usepackage[square,numbers]{natbib}
\usepackage{cleveref}

\usepackage{authblk}

\usepackage[textsize=footnotesize]{todonotes}

\definecolor{bred}{rgb}{0.8,0,0}

\hypersetup{colorlinks,linkcolor={blue},citecolor={bred},urlcolor={blue}}

\newcommand{\E}{\mathbb{E}}
\newcommand{\tfls}{\theta^\lambda_{\lfloor s \rfloor}}

\newcommand{\Vfls}{V^\lambda_{\lfloor s \rfloor}}

\newcommand{\he}{h_{MY,\epsilon}}
\newcommand{\fg}{f_{tam,\gamma}}
\newcommand{\hg}{h_{tam,\gamma}}
\newcommand{\slg}{S_{\lambda,\gamma}}
\newcommand{\tflt}{\theta^\lambda_{\lfloor t \rfloor}}
\newcommand{\zflt}{\zeta^\lambda_{\lfloor t \rfloor}}
\newcommand{\Vflt}{V^\lambda_{\lfloor t \rfloor}}
\newcommand{\tn}{\bar{\theta}^\lambda_n}
\newcommand{\vn}{\bar{V}^\lambda_n}
\newcommand{\xn}{\bar{x}^\lambda_n}
\newcommand{\yn}{\bar{Y}^\lambda_n}

\newcommand{\q}{Q^{\lambda,n}}
\newcommand{\p}{p^{\lambda,n}}

\title{Kinetic Langevin MCMC Sampling Without Gradient Lipschitz Continuity - the Strongly Convex Case }

\author[1]{Tim Johnston}
\author[1]{Iosif Lytras}
\author[1, 2, 3]{Sotirios Sabanis}

\affil[1]{\footnotesize School of Mathematics, The University of Edinburgh, Edinburgh, UK.}
\affil[2]{\footnotesize The Alan Turing Institute, London, UK.}
\affil[3]{\footnotesize National Technical University of Athens, Athens, Greece.}

\begin{document}

\maketitle
\begin{abstract}
In this article we consider sampling from log concave distributions in Hamiltonian setting, without assuming that the objective gradient is globally Lipschitz. We propose two algorithms based on monotone polygonal (tamed) Euler schemes, to sample from a target measure, and provide non-asymptotic 2-Wasserstein distance bounds between the law of the process of each algorithm and the target measure. Finally, we apply these results to bound the excess risk optimization error of the associated optimization problem.
\end{abstract}
\section{Introduction}
We are interested in non-asymptotic estimates for samples taken from distributions of the form
\begin{equation}\label{eq-meas1}
\mu_{\beta}(\mathrm{d} \theta) \propto \exp (-\beta u(\theta)) \mathrm{d} \theta,
\end{equation}
where $u:\mathbb{R}^d \rightarrow \mathbb{R}_+$ and $\beta>0.$
This problem has many applications in optimization and particularly, in machine learning as it is known that for large values of the temperature parameter $\beta>0$, the target measure $\mu_\beta$ concentrates around the global minimizers of the function $u$ (\cite{Hwang}). An efficient way of sampling from the target measure is through simulating the associated (overdamped) Langevin SDE
\begin{equation}\label{eq-over-Langevin}
\begin{aligned}
\mathrm{d} L_{t}&=-\nabla u\left(L_{t}\right) \mathrm{d} t+\sqrt{\frac{2}{\beta}} \mathrm{d} B_{t},\quad t\geq 0,
\\L_0 &=\theta_0 ,
\end{aligned}
\end{equation}
exploiting the fact that this SDE admits the target measure as its unique invariant measure. Inspired by this observation Langevin-based algorithms have become a popular choice  for sampling from distributions in high dimensional spaces.

The non-asymptotic convergence rates of algorithms based on \eqref{eq-over-Langevin} have been studied extensively in recent literature. For the setting of deterministic gradients some recent key references are \cite{unadjusted}, \cite{cheng}, \cite{TULA}, \cite{hola}, whilst for stochastic gradient algorithms recent important contributions in the convex case are \cite{brosse2018promises}, \cite{dk}, \cite{convex} and in the non-convex setting \cite{Raginsky}, \cite{nonconvex} \cite{chau2019stochastic}. More recently work has been done assuming only local conditions in \cite{zhang2019nonasymptotic} and \cite{TUSLA}.

An alternative to methods based on overdamped Langevin SDE is the class of algorithms based on the underdamped Langevin SDE, which has also gained significant attention in recent years. Let $\gamma>0$. The underdamped Langevin SDE is given by

\begin{equation}\label{eq-underlang}
\begin{aligned}
&\mathrm{d} \tilde{V}_{t}=-\gamma \tilde{V}_{t} \mathrm{~d} t-\nabla u\left(\tilde{\theta}_{t}\right) \mathrm{d} t+\sqrt{\frac{2 \gamma}{\beta}} \mathrm{d} B_{t}, \\
&\mathrm{~d} \tilde{\theta}_{t}=\tilde{V}_{t} \mathrm{~d} t,\quad t>0.
\end{aligned}
\end{equation}
with initial conditions
\[\begin{aligned}
\tilde{V}_0&=V_0,\\
\tilde{\theta}_0&=\theta_0,
\end{aligned}\]
where $\{B_t\}_{t \geq 0}$ is Brownian motion adapted to its natural filtration $\mathbb{F}:=\{\mathcal{F}_t\}_{t\geq 0}$
and $\left\{\left(\tilde{\theta}_{t}, \tilde{V}_{t}\right)\right\}_{t \geq 0}$ are called position and momentum process respectively. Similarly to the overdamped Langevin SDE, this diffusion can be used as both an MCMC sampler and non-convex optimizer, since under appropriate conditions, the Markov process $\left\{\left(\tilde{\theta}_{t}, \tilde{V}_{t}\right)\right\}_{t \geq 0}$ has a unique invariant measure given by
\begin{equation}\label{eq-pibeta}
{\pi}_{\beta}(\mathrm{d} \theta, \mathrm{d} v) \propto \exp \left(-\beta\left(\frac{1}{2}|v|^{2}+u(\theta)\right)\right) \mathrm{d} \theta \mathrm{d} v.
\end{equation}

Consequently, the marginal distribution of \eqref{eq-pibeta} in $\theta$ is precisely the target measure defined in \eqref{eq-meas1}. This means that sampling from \eqref{eq-pibeta} in the extended space and then keeping the samples in the $\theta$ space defines a valid sampler for the density given by \eqref{eq-meas1}.

The underdamped Langevin MCMC has shown improved convergence rates in the case of convex and Lipschitz potential $u$, e.g see \cite{cheng2018underdamped},  \cite{Dalaylan}. In the non-convex setting important work has been done in \cite{Deniz}, \cite{Gao}, \cite{chau2019stochastic}. The current article extends these results by introducing a novel way to sample from distributions associated with objective functions with non-global Lipschitz gradients, utilising a decomposition proposed in \cite{johnston2021strongly}. To our knowledge this relaxation of the global gradient Lipschitz assumption is a novelty in underdamped (HMC) sampling, and can be seen as an analogous of the algorithms in \cite{TULA} and \cite{TUSLA} which are examples of taming sampling algorithms in the overdamped case.

{\textbf{Notation}}. For an integer $d \geq 1$, the Borel sigma-algebra of $\mathbb{R}^{d}$ is denoted by $\mathcal{B}\left(\mathbb{R}^{d}\right)$. We write the dot product on $\mathbb{R}^d$ as $\langle\cdot, \cdot\rangle$, whilst $|\cdot|$ denotes the associated vector norm and $||\cdot||$ denotes the Euclidean matrix norm. Moreover we denote by $C^{m}$ the space of $m$-times continuously differentiable functions, and by $C^{1,1}$ the space of continuously differentiable functions with Lipschitz gradient. The set of probability measures defined on the measurable space $\left(\mathbb{R}^{d}, \mathcal{B}\left(\mathbb{R}^{d}\right)\right)$ is denoted by $\mathcal{P}\left(\mathbb{R}^{d}\right)$. For an $\mathbb{R}^{d}$ -valued random variable, $\mathcal{L}(X)$ and $\mathbb{E}[X]$ are used to denote its law and its expectation respectively. Note that we also write $\mathbb{E}[X]$ as $\mathbb{E} X$. For $\mu, \nu \in \mathcal{P}\left(\mathbb{R}^{d}\right)$, let $\mathcal{C}(\mu, \nu)$ denote the set of probability measures $\Gamma$ on $\mathcal{B}\left(\mathbb{R}^{2 d}\right)$ so that its marginals are $\mu, \nu$. Finally, for $\mu, \nu \in \mathcal{P}\left(\mathbb{R}^{d}\right)$, the Wasserstein distance of order $p \geq 1$, is defined as
$$
W_{p}(\mu, \nu):=\inf _{\Gamma \in \mathcal{C}(\mu, \nu)}\left(\int_{\mathbb{R}^{d}} \int_{\mathbb{R}^{d}}\left|\theta-\theta^{\prime}\right|^{p} \Gamma\left(\mathrm{d} \theta, \mathrm{d} \theta^{\prime}\right)\right)^{1 / p},
$$
and for an open $\Omega\subset \mathbb{R}^d$ the Sobolev space $H^1(\Omega)$ as
$$
\begin{aligned}
H^{1}(\Omega):=&\left\{u \in L^{2}(\Omega): \exists g_{1}, g_{2}, \ldots, g_{N} \in L^{2}(\Omega),\right. \text { such that }\\
&\left.\int_{\Omega} u \frac{\partial \phi}{\partial x_{i}}=-\int_{\Omega} g_{i} \phi, \quad  \phi \in C_{c}^{\infty}(\Omega),  i: 1 \leq i \leq N\right\},
\end{aligned}
$$
where $C^\infty_c(\Omega)$ denotes the space of infinitely times differentiable and compactly supported functions.
\section{Overview of main results}
\subsection{Assumptions and marginal distributions of interest}
Let $u \in C^2(\mathbb{R}^d)$  be a non-negative objective function let $h: =\nabla u$. Then we assume the following:

\newtheorem{Ass1}{Assumption}[section]
\begin{Ass1}(strong convexity). \label{A1}
There exists an $m>0$ such that
\[\langle h(x)-h(y),x-y\rangle \geq 2m |x-y|^2 \quad  x,y \in \mathbb{R}^d.\]
\end{Ass1}

\newtheorem{Ass2}[Ass1]{Assumption}
\begin{Ass2}\label{A2}

(Local Lipschitz continuity). There exists $l>0$ and $L>0$ such that
\[|h(x)-h(y)|\leq L (1+|x|+|y|)^l |x-y| \quad  x,y\in \mathbb{R}^d.\]
\end{Ass2}
\newtheorem{Ass3}[Ass1]{Assumption}
\begin{Ass3}\label{A3}
There holds
\[\E |\theta_0|^{2k} + \E |V_0|^{2k}<\infty \quad \forall k \in \mathbb{N}.\]
\end{Ass3}
\newtheorem{Remark}[Ass1]{Remark}
\begin{Remark}
The moment requirement in Assumption \ref{A3} can be further relaxed. This condition is chosen for notational convenience to improve the readability of the article.
\end{Remark}
We denote by $x^*$ the minimizer of $u$. Let us now define the marginal distributions of $\pi_\beta$ which is given in \eqref{eq-pibeta}, and plays a pivotal role in our analysis.
\newtheorem{measures}[Ass1]{Definition}
\begin{measures}
Let  $\mu_\beta(x):=\frac{e^{-\beta u(x)}}{\int_{\mathbb{R}^d} e^{-\beta u(x)}dx}$
and $k_\beta(y):=\frac{e^{-\beta\frac{|y|^2}{2}}}{\int_{\mathbb{R}^d} e^{-\beta\frac{|y|^2}{2}}dy} .$\\
Then $\pi_\beta$ is the product measure with density $\mu_\beta\otimes k_\beta$.
\end{measures}
\subsection{Statement of main results}
Let $\{(\tn,\vn)\}_{n\geq0}$, $\{(\xn,\yn)\}_{n\geq 0}$ be the processes generated by the tKLMC1 and tKLMC2 algorithms respectively as given in \eqref{eq-discrete} and \eqref{eq-tKLMC2} respectively. Let us state the main theorems regarding their $W_2$ convergence to the invariant measure $\pi_\beta.$
\newtheorem{theorem}[Ass1]{Theorem}
\begin{theorem}\label{Final rate}
Let $q>0$,$\epsilon>0$ and Assumptions \ref{A1}-\ref{A3} hold. Let $K$ be the Lipschitz constant of the gradient of the $\epsilon$ -Moreau-Yosida approximation of $u$. Then there exists $\dot{C}>0$ such that, for $\gamma\geq \max\{ \sqrt{\frac{K+m}{\beta}},K,32,\frac{48(2m+1)^2}{m}\}$, and $\lambda< \gamma^{-1}$,
\[W_2(\mathcal{L}(\bar{\theta}^\lambda_n,\bar{V}^\lambda_n),\pi_\beta)\leq \dot{C}\left( \sqrt{\lambda}\gamma^{\frac{3}{2}}+\gamma^{-q}+\sqrt{2}e^{-\frac{\lambda m}{\beta \gamma}n}W_2(\mathcal{L}(\theta_0,V_0),\pi_\beta)\right)+\epsilon, \]
where $\dot{C}$ has $(2l+2)2q$ dependence on $\frac{d}{\beta}$.
\end{theorem}
\newtheorem{Concluding rate}[Ass1]{Remark}
\begin{Concluding rate}
Setting $\epsilon=\lambda^\frac{1}{5}$ and allowing the worst-case scenario that $K=\frac{1}{\epsilon}$ then, setting $\gamma=\mathcal{O}(\frac{1}{\epsilon})$ and $q=5$ yields
\[W_2(\mathcal{L}(\bar{\theta}^\lambda_n,\bar{V}^\lambda_n),\pi_\beta)\leq 4\dot{C}\left(\lambda^\frac{1}{5}+\sqrt{2}e^{-\frac{\lambda^\frac{6}{5} m}{\beta }n}W_2(\mathcal{L}(\theta_0,V_0),\pi_\beta)\right) \quad \forall n \in \mathbb{N}.\]
\end{Concluding rate}
\begin{theorem}\label{Final rate 2}
Let $\epsilon>0, q>0$ and Assumptions \ref{A1}-\ref{A3} hold.
Let $K$ be the Lipschitz constant of the gradient of the $\epsilon$ -Moreau-Yosida approximation of $u$.
There exists $C$ such that for $\gamma\geq \sqrt{\frac{2K}{\beta}}$ and $\lambda \leq \mathcal{O}(\gamma^{-5})$,
\[
\begin{aligned}
W_2\left(\mathcal{L}(\xn,\yn),\pi_\beta\right)\leq C (\lambda \gamma^2 K +\gamma \epsilon +\gamma^{-{q}+2})+\sqrt{2}e^{-\frac{m}{\gamma \beta}\lambda n}W_2(\mathcal{L}(\theta_0,V_0),\pi_\beta),\end{aligned}\]
where ${C}$ has $(2l+2)2q$ dependence on $\frac{d}{\beta}$.
\end{theorem}
\newtheorem{RemarkKLMC2}[Ass1]{Remark}
\begin{RemarkKLMC2}
The worst case bounds, when $K=\frac{2}{\epsilon}$ can be derived by setting $\gamma=\sqrt{\frac{2}{\beta}}\sqrt{\frac{1}{\epsilon}}$,
\[W_2\left(\mathcal{L}(\xn,\yn),\pi_\beta\right)\leq 2C(\lambda  \gamma^{4} +\gamma^{-1}+\gamma^{-\frac{q}{2}+1})+\sqrt{2}e^{-\frac{m}{\gamma \beta}\lambda n}W_2(\mathcal{L}(\theta_0,V_0),\pi_\beta)\]
so if $\lambda=\mathcal{O}(\gamma^{-5})$
\[W_2\left(\mathcal{L}(\xn,\yn),\pi_\beta\right)\leq \bar{C} \lambda^\frac{1}{5} +\sqrt{2}e^{-\frac{m}{\gamma \beta}\lambda n}W_2(\mathcal{L}(\theta_0,V_0),\pi_\beta)\]
\end{RemarkKLMC2}
The two algorithms are two alternative tamed Euler discretizations of the underdamped Langevin dynamics. They can be seen as the tamed analogous (in the deterministic gradient case) of the SGHMC1 and SGHMC2 algorithms in \cite{gao2021global}.
From Theorem 1 and Theorem 2 one can see that the convergence rates are fairly comparable although the KLCM2 algortithm has superior $\lambda$- dependence on the first term, which leads to improved resutls in the Lipschitz case.
In addition, we show that for some $R_0$ depending on $u$ which is given in the proof, the process $\bar{\theta}_n^\lambda$ generated by the tKLMC1 algorithm  (similar results can be obtained for the process $\xn$ built by the tKLMC2 algortihm),
can be used for the optimization of the excess risk of the associated optimization problem, when one estimates the minimizer of $u$ for $\beta \rightarrow \infty$.
\begin{theorem}
Let Assumptions \ref{A1}-\ref{A3} hold.
and $R_0=\sqrt{\frac{2u(0)}{m}} + \sqrt{\frac{(u(0)+1)d}{\beta}} +1.$ Then,
\[\E u\left(\bar{\theta}_n^\lambda\mathds{1}_{B(0,R_0)}(\bar{\theta}_n^\lambda)\right)-u(x^*)\leq C'W_2(\mathcal{L}(\bar{\theta}^\lambda_n,\bar{V}_n^\lambda),\pi_\beta) +2u(0)W_2(\mathcal{L}(\bar{\theta}^\lambda_n,\bar{V}_n^\lambda),\pi_\beta)+\frac{2d}{m\beta}\]
where $C'$ is a constant with polynomial dependence on $d$.
\end{theorem}
\newtheorem{remarkopt}[Ass1]{Remark}
\begin{remarkopt}
Although we are in a strongly convex setting, the fact that the gradient is not assumed to be globally Lipschitz provides
additional complexity. To the best of the authors' knowledge, there is little in the literature for this class of problems, even for gradient descent algorithms which make no use of Langevin dynamics.
\end{remarkopt}
\subsection{Overview of contributions and comparison with existing literature}
In this article we study the problem of sampling from strongly log-concave distributions with superlinearly-growing gradients. To our knowledge, the current literature for sampling via kinetic Langevin (underdamped) algorithms assumes only global Lipschitz continuity of the objective gradient, which is in many cases restrictive. The main contribution of this article, therefore, is the weakening of the gradient-Lipschitz assumption, replacing it with Assumption \ref{A2} which is considerably weaker. This means our results can be applied to the case of superlinearly growing gradients, for instance for neural networks as in \cite{TUSLA}.

A natural comparison can be made with \cite{Dalaylan}, which also considers sampling from strongly convex distributions with kinetic Langevin algorithms in the Lipschitz setting. The converge rate of our algorithm is worse, which is unsurprising since there are no known contraction rates in the non-gradient Lipschitz setting.

To address this problem we first sample from the approximating Moreau-Yosida measure (since the Moreau Yosida gradient is Lipschitz). In the general case one has that the gradient of the Moreau-Yosida regularization $\he$ is $\frac{1}{\epsilon}$-Lispchitz (although in practice could be considerably lower). Since the proof of the convergence results relies upon the contraction results of \cite{Dalaylan}, this imposes a technical restriction, connecting the parameter $\gamma$ to the worst case Lipschitz constant,  $\frac{1}{\epsilon}$.
Another explanation for the suboptimality of the convergence rates, is the fact that instead of sampling with an algorithm that involves the Moreau-Yosida gradient (since we first sample from the Moreau-Yosida measure) we chose a tamed scheme of the original gradient $h$. The reason for this is that an algorithm involving the Moreau-Yosida gradient, would impose many practical difficulties in applications as it would require a solution of an optimization problem in each step.
The "taming error'' (i.e second term in Theorem 1 and third term in Theorem 3) can be seen as the price we pay for this adjustment.
Despite all these technical restrictions, the fact we were able to establish convergence in the non-Lipschitz setting shows the advantage of a more sophisticated taming scheme like the one used in the article, making a first promising step into the direction of sampling from disrtibution with superlinear log-gradients in the kinetic Langevin setting.

In the case where $h$ is $M$-Lipschitz one may choose $\gamma\geq \sqrt{\frac{M+m}{\beta}}$ independently of $\epsilon$, so that setting $\epsilon=\sqrt{\lambda}$ yields a convergence rate of $\mathcal{O}(\sqrt{\lambda}+ \gamma^{-q})$ for tKLMC1 and $\mathcal{O}({\lambda}+ \gamma^{-q+2})$ for tKLMC2 where the term $\gamma^{-q}$ is only a result of the 'taming' (and therefore is not needed if the gradient is growing linearly).\\ An interesting comparison is the one with the analogous result in the overdamped case in \cite{TULA}. One significant advantage of our tamed scheme compared to the one used in \cite{TULA} is that it doesn't require Assumption H2 to prove the necessary moment bounds.
Furthermore, although the achieved convergence rates are worse than those in \cite{TULA} with respect to the stepsize, the dependence of the dimension of our constants is polynomial, whilst in \cite{TULA} it is exponential (possibly due to overestimations as stated in Remark 16).\\ Finally , our algorithms could perform well for optimization problems as an alternative to gradient descent by picking $\beta$ large enough. Since there are few works in the non-gradient Lipschitz setting this could possibly be useful in applications. For example, compared with the work in \cite{zhang2020first}, our complexity bounds have better dependence on the growth parameter $l$ in \ref{A2}, which theoretically makes it more suitable for higher order problems.
\section{First remarks and premiliminary statements}
We start our analysis by stating some preliminary results derived from a assumptions \ref{A1}-\ref{A2}.
\newtheorem{effect of conv}[Ass1]{Remark}
\begin{effect of conv}\label{effect of conv}
Since $h$ is the gradient of a strongly convex function $u\geq 0$ one obtains that
\[u(0)\geq u(x) + \langle h(x), 0-x\rangle+ {m}|x|^2.\]
Then it can be easily deduced that


\begin{equation}\label{eq-diss1}
    \langle h(x),x\rangle \geq \frac{m}{2}|x|^2 -u(0), \quad  x \in \mathbb{R}^d.
\end{equation}
\end{effect of conv}
\newtheorem{minim bound}[Ass1]{Remark}
\begin{minim bound}
Setting $x=x^*$ in \eqref{eq-diss1} yields
\[\frac{m}{2}|x^*|^2-u(0)\leq 0,\]
which implies that
\[x^* \in \bar{B}(0,\sqrt{\frac{2u(0)}{m}}).\]
\end{minim bound}
As a consequence of these results, we obtain useful bounds on the moments of the invariant measure.
\newtheorem{over-damp second moment}[Ass1]{Lemma}
\begin{over-damp second moment}[ \cite{Raginsky}, Lemma 3.2]\label{over-damp second moment}
Let Assumptions \ref{A1}-\ref{A3} hold.
Let ${L}_t$ be given by the overdamped Langevin SDE in \eqref{eq-over-Langevin} with initial condition $L_0$.
Then
\begin{equation}\label{eq-oversecondm}
    \E |{L}_t|^2 \leq \E|L_0|^2e^{-mt} + 2\frac{u(0)+d/\beta}{m}(1-e^{-mt}),
\end{equation}
for $a$ and $b$ as above.
\end{over-damp second moment}
\begin{proof}
Using Ito's formula on $Y_t=|L_t|^2$ one deduces
\[dY_t=-2\langle L_t,h(L_t)\rangle dt+\frac{2d}{\beta}+\sqrt{\frac{8}{\beta}}\langle L_t,dB_t\rangle.\]
Applying Ito's formula on $e^{2mt}|L_t|^2$
\[d(e^{mt}Y_t)=-2e^{mt}\langle L_t,h(L_t)\rangle dt +m e^{mt}Y_tdt +\frac{2d}{\beta} e^{mt}dt+\sqrt{\frac{8}{\beta}} e^{mt}\langle L_t,dB_t\rangle\]
so after rearranging
\begin{equation}\label{Ito-mom}
    Y_t=e^{-mt}Y_0-2\int_0^t e^{-m(t-s)}\langle L_s,h(L_s)\rangle ds+ m \int_0^t e^{ m(s-t)} Y_s \mathrm{d} s+\frac{2d}{m \beta}\left(1-e^{- m t}\right)+M_t,
\end{equation}
where $M_t$ is an $\mathbb{F}_t$ martingale (as one can show by applying standard stopping time techniques initially to show the finiteness of relative moments without the use of integrated factors).
Using the dissipativity conditon on the second term yields
\[\begin{aligned}
 -2\int_0^t e^{-m(t-s)}\langle L_s,h(L_s)\rangle ds&\leq 2 \int_0^t e^{ m(s-t)}(u(0)-\frac{m}{2} Y_s) \mathrm{d} s\\&\leq \frac{2u(0)}{m}\left(1-e^{- m t}\right)-m\int_0^t e^{-m(t-s)} Y_s ds.
\end{aligned}\]
Inserting this into
\eqref{Ito-mom} leads to
\[Y_t=e^{-mt}Y_0+\frac{2u(0)}{m}\left(1-e^{- m t}\right)+\frac{2d}{m \beta}\left(1-e^{- m t}\right)+M_t.\]
Taking expectations and using the martingale property completes the proof.
\end{proof}
\newtheorem{inv-measure bound}[Ass1]{Lemma}
\begin{inv-measure bound}\label{inv-measure bound}
Let Assumptions \ref{A1}-\ref{A3} hold.
If $Y$ is a random variable such that $\mathcal{L}(Y)={\mu}_{\beta}$ then,
\[\E |Y|^2\leq  \frac{2}{m} (u(0) +\frac{d}{\beta}). \]
As a result, \begin{equation}
    \E |Y|\leq \sqrt{\frac{2}{m} (u(0) +\frac{d}{\beta})}.
\end{equation}
\end{inv-measure bound}
\begin{proof}
Since the solution of the Langevin SDE \eqref{eq-over-Langevin} converges in $W_2$ distance to $\mu_\beta$ and\\ $\sup_{t\ge0}\E |\bar{L}_t|^2<\infty$ due to \eqref{eq-oversecondm}, this also implies the convergence of second moments. Therefore, if $Y$ is random variable such that $\mathcal{L}(W)=\mu_\beta$ by Lemma \ref{over-damp second moment} there holds that
\[\E |Y|^2=\lim _{t\rightarrow \infty} \E |\bar{L}_t|^2 \leq \frac{2}{m} (u(0) +\frac{d}{\beta}) .\]
\end{proof}
\newtheorem{harge}[Ass1]{Theorem}
\begin{harge}\label{harge}[\cite{harge2004convex}, Theorem 1.1]
Let $X$ follow $\mathcal{N}(m_0,\Sigma)$ with density $\phi$, and let $Y$ have density $\phi f$ where $f$ is a log-concave function. Then for any convex map $g$ there holds:
\[\E [g\left(Y-\E Y\right)]\leq \E [g\left(X-\E X\right)] .\]
\end{harge}
\newtheorem{Inv mom}[Ass1]{Proposition}
\begin{Inv mom}\label{Inv mom}
Let $p\geq 2$ and let $Y$ be a random variable such that $\mathcal{L}(Y)=\mu_\beta.$ Then
\[ \E|Y|^p\leq   2^{p-1} ((\frac{d}{\beta m})^{\frac{p}{2}} (1+p/d)^{\frac{p}{2}-1} +\left (\frac{2}{m} (u(0) +\frac{d}{\beta}) \right )^{p/2}:= C_{\mu_\beta,p}. \]
\end{Inv mom}

\begin{proof}
Since $e^{-\beta u(x)}=e^{-\beta( u- \frac{m}{2}|x|^2)} e^{-\beta \frac{m}{2}|x|^2}  $ and since the function $u-\frac{m}{2}|x|^2$ is convex (due to strong convexity of $u$) the assumptions of Theorem \ref{harge} are valid for $X$ which has distribution $\mathcal{N}(0,\frac{1}{\beta m}I_d)$, and for $Y$ with density $\mu_\beta$. Since the function $g: x\rightarrow |x|^p$ is convex applying the result of Theorem \ref{harge} leads to
\begin{equation}\label{pol mom}
    \E |Y-\E Y|^p\leq \E |X|^p=(\frac{d}{\beta m})^{\frac{p}{2}} \frac{\Gamma((d+p)/2)}{\Gamma(d/2)(d/2)\frac{p}{2}}\leq (\frac{d}{\beta m})^{\frac{p}{2}} (1+p/d)^{\frac{p}{2}-1}.
\end{equation}
Combining \eqref{pol mom} and the result of Lemma \ref{inv-measure bound} yields
\begin{equation}
    \E|Y|^p\leq 2^{p-1}\left( \E|Y-\E Y|^p + (\E |Y|)^p\right)\leq 2^{p-1} ((\frac{d}{\beta m})^{\frac{p}{2}} (1+p/d)^{\frac{p}{2}-1} +\left (\frac{2}{m} (u(0) +\frac{d}{\beta}) \right )^{p/2},
\end{equation}
which completes the proof.
\end{proof}
\section{A brief introduction to Moreau-Yosida regularization}
It is well-known that in order to construct a Langevin-type algorithm to efficiently sample from a target measure, one requires a contraction property. However, to the best of the authors' knowledge, Lipschitz continuity of the objective function gradient (drift coefficient in the case of the underdamped Langevin SDE) is a requirement in the literature in order to obtain contraction rates in the underdamped case. The purpose of this section, therefore, is to construct a $C^{1,1}$ approximation of $u$ that inherits the convexity properties of $u$, and for which the associated target measure is close to $\pi_\beta$.

Let us present some facts about the Moreau-Yosida regularization which are central to the subsequent analysis.
For the interested reader, key references are \cite{rockafellar2009variational}, \cite{durmus2018efficient}, \cite{lemarechal1997practical}. For an arbitrary lower semicontinuous function $g:\mathbb{R}^d\rightarrow (-\infty, \infty]$ and $\epsilon>0$ the $\epsilon$-Moreau Yosida regularization $g^\epsilon:\mathbb{R}^d\rightarrow (-\infty, \infty]$ is given by
\[g^\epsilon(x):=\inf_{y\in \mathbb{R}^d} \{g(y)+ \frac{1}{2\epsilon}|x-y|^2\} .\]
The first thing to note is that in the case that $g$ is convex
\begin{equation}\label{eq-minimizers}
\begin{aligned}
      \arg\min g^\epsilon&=\arg\min g,\\
      \min g^\epsilon&=\min g.
\end{aligned}
\end{equation}
Furthermore, if one defines the convex conjugate of $g$  as
$$
g^{*}(v):=\sup _{x \in dom(g)}\{\langle v, x\rangle-g(x)\}.
$$
then it can be shown that
\begin{equation}\label{eq-dual}
    g^\epsilon=\left (g^* + \frac{\epsilon}{2}|\cdot|^2\right)^*.
\end{equation}
The approximation $g^\epsilon$ inherits the convexity property of $g$ and is continuously differentiable with $\frac{1}{\epsilon}$ Lipschitz gradient, i.e
\[|\nabla g^\epsilon(x)-\nabla g^\epsilon(y)|\leq \frac{1}{\epsilon}|x-y|, \quad  x,y \in \mathbb{R}^d, \]
and $\nabla g^\epsilon$ is given by
\begin{equation}\label{eq-prox}
\nabla g^\epsilon(x)=\frac{1}{\epsilon}(x-\operatorname{prox}_g^\epsilon(x)),
\end{equation}
where
\[\operatorname{prox}_{{g}}^{\epsilon}(x):=\underset{y \in \mathbb{R}^{d}}{\arg \min }\left\{{g}(y)+\frac{1}{2\epsilon}|x-y|^{2}\right\},\]
is a continuous mapping. In addition, for every $x \in \mathbb{R}^d$, the function $\epsilon \mapsto g_\epsilon(x)$ is decreasing, and one has the bound $g^\epsilon\leq g$ and the convergence result
\[\lim_{\epsilon \rightarrow 0} g^\epsilon(x)=g(x), \quad  x\in \mathbb{R}^d  .\]
Finally $g^\epsilon$ inherits the differentiability and convexity properties of $g$, and if $g\in C^1$ then one has the relation
\begin{equation} \label{gradient connection}
    \nabla g^{\epsilon}(x)=\nabla g (\operatorname{prox}_{{g}}^{\epsilon}(x)).
\end{equation}
In particular, the following result regarding the second order differentiability of $g^\epsilon$ is important for our analysis.
\newtheorem{Second order differentiability}[Ass1]{Theorem}
\begin{Second order differentiability}\label{Second order}[\cite{planiden2019proximal}, Theorem 3.12]
Let $g$ be a proper lower semicontinuous convex function. Then if $g\in \mathcal{C}^2$ one has that $g^\epsilon \in \mathcal{C}^2$.
\end{Second order differentiability}
\newtheorem{strong conv of Yos}[Ass1]{Lemma}
\begin{strong conv of Yos}\label{strong conv}
Let $g$ be a proper lower semicontinuous $m$ strongly convex function. Then, if $\epsilon<\frac{1}{m}$, $g^\epsilon$ is an $\frac{m}{2}$ strongly convex function.
\end{strong conv of Yos}
\begin{proof}
Proposition 12.60 in \cite{rockafellar2009variational} implies that
\begin{equation}\label{eq-variational}
    \nabla (g^\epsilon)^*(x)=\nabla g^*(x) +\epsilon x .
\end{equation}
Furthermore, since $g$ is $m$- strongly convex we have that $\nabla g^*$ is $\frac{1}{m}$-Lipschitz. Hence, by the assumption $\epsilon<\frac{1}{m}$ it follows from \eqref{eq-variational} that $\nabla (g^\epsilon)^*$ is  $\frac{2}{m}$-Lipschtiz. Then utilising the same argument from \cite{rockafellar2009variational}, one obtains that $g^\epsilon$ is $\frac{m}{2}$ strongly convex.
\end{proof}
\subsection{Finding the right approximation for the objective gradient}
Let us now introduce the function $u_{MY,\epsilon}:\mathbb{R}^d\rightarrow \mathbb{R}_+$ given as the Moreau-Yosida regularization of $u$, that is
\begin{equation}
u_{MY,\epsilon}(x):=\inf \{{u}(y) +\frac{1}{2\epsilon}|x-y|^2\},
\end{equation}
and denote by $h_{MY,\epsilon}=\nabla u_{MY,\epsilon}$ its gradient. Let $\lambda>0$ be the stepsize of the proposed algorithm, and let $\epsilon=\epsilon(\lambda)$ be a function of $\lambda$. By Theorem \ref{Second order} the function $u_{MY,\epsilon}$ is twice continuously differentiable, and by Lemma \ref{strong conv} the function $u_{MY,\epsilon}$ is strongly convex with parameter $m$ if $\epsilon<\frac{1}{m}$. Moreover, $h_{MY,\epsilon}$ is $\frac{1}{\epsilon}$ Lipschitz continuous, and as a result
\begin{equation}\label{eq-Hess}
m I_d\leq Hess(u_{MY,\epsilon})\leq \frac{1}{\epsilon} I_d.
\end{equation}
Since $x^*$ is the unique minimizer of $u$ and due to \eqref{eq-minimizers},
\begin{equation}\label{eq-comp0}
   0\leq u(x^*)=u_{MY,\epsilon}(x^*)\leq u_{MY,\epsilon}(x) ,
\end{equation}
for every $x \in \mathbb{R}^d$. Another key property is the inequality
\begin{equation}\label{eq-inequality}
    u_{MY,\epsilon}(x)\leq u(x),
\end{equation}
which immediately implies that
\begin{equation}\label{eq-int ineq}
 {Z^\epsilon_\beta} := \int_{\mathbb{R}^d} e^{-\beta u_{MY,\epsilon}(x)}dx\geq \int_{\mathbb{R}^d} e^{-\beta u(x)}dx:=Z_\beta.
\end{equation}
Finally, one observes that
\begin{equation}\label{eq-dissipative}
    \langle h_{MY,\epsilon}(x),x\rangle \geq \frac{m}{2}|x|^2-u_{MY,\epsilon}(0)\geq \frac{m}{2}|x|^2-u(0).
\end{equation}
\newtheorem{Measures}[Ass1]{Definition}
\begin{Measures}
Let $0<\epsilon<\frac{1}{m}$. We define \[\mu^\epsilon_\beta(x):=\frac{e^{-\beta u_{MY,\epsilon}(x)}}{\int_{\mathbb{R}^d} e^{-\beta u_{MY,\epsilon}(x)}dx},\] and
\[\pi_\beta^\epsilon(x,y) := \frac{e^{-\beta u_{MY,\epsilon}(x)-\frac{\beta}{2}|y|^2}}{\int_{\mathbb{R}^d} e^{-\beta u_{MY,\epsilon}(x)-\frac{\beta}{2}|y|^2}dxdy}.\]
\end{Measures}

Having stated the key properties of $u_{MY,\epsilon}$, let us consider the properties of the respective target measure with respect to the $W_2$ distance.


\newtheorem{lemmanormineq}[Ass1]{Lemma}
\begin{lemmanormineq}\label{lemmanormineq}
Let $\epsilon>0$. Then one obtains\[|\nabla u_{MY,\epsilon}(x)|\leq |\nabla u(x)|,\] for every $x\in\mathbb{R}^d$.
\end{lemmanormineq}
\begin{proof}
Recall that by \eqref{gradient connection} and \eqref{eq-prox} one has $h(\operatorname{prox}_u^\epsilon(x))=\frac{1}{\epsilon}(x-\operatorname{prox}_u^\epsilon(x))$, where we have defined $h:=\nabla u.$ When $x=\operatorname{prox}_u^\epsilon(x)$ the result holds trivially. So let us assume $x\not = \operatorname{prox}_u^\epsilon(x)$. First of all by the strong convexity there holds \[\langle h(x)-h(\operatorname{prox}_u^\epsilon(x)),x-\operatorname{prox}_u^\epsilon(x)\rangle \geq m |x-\operatorname{prox}_u^\epsilon(x)|^2,\] which implies that

\[\begin{aligned}|h(x)||x-\operatorname{prox}_u^\epsilon(x)|&\geq \langle h(x), x-\operatorname{prox}_u^\epsilon(x)\rangle \\&\geq m|x-\operatorname{prox}_u^\epsilon(x)|^2 +\langle h(\operatorname{prox}_u^\epsilon(x)),x-\operatorname{prox}_u^\epsilon(x)\rangle\\&=(\frac{1}{\epsilon}+m)|x-\operatorname{prox}_u^\epsilon(x)|^2\\&\geq \frac{1}{\epsilon}|x-\operatorname{prox}_u^\epsilon(x)|^2.\end{aligned}\]
This means that since $x\not=\operatorname{prox}_u^\epsilon(x)$
there holds, due to \eqref{gradient connection},
\[|h_{MY,\epsilon}(x)|=\frac{1}{\epsilon}|x-\operatorname{prox}_u^\epsilon(x)|\leq |h(x)|,\] which completes the proof.
\end{proof}
\newtheorem{gradient dist}[Ass1]{Lemma}
\begin{gradient dist}\label{gradient dist}
Let $0<\epsilon<\frac{1}{m}$. Then, $\forall x \in \mathbb{R}^d,$
\[|h(x)-\he(x)|\leq 2^{2l+2}L(1+|x|+\sqrt{R})^{2l+2}\epsilon\]
where $R=\sqrt{\frac{2u(0)}{m}}).$
\end{gradient dist}
\begin{proof}
Using \eqref{gradient connection}, Assumption \ref{A2} and Lemma \ref{lemmanormineq} one obtains, \begin{equation}\label{eq-grad-dist}
\begin{aligned}
|h(x)-h_{MY,\epsilon}(x)|&\leq |h(x)-h(prox_u^\epsilon(x)|\\&\leq L(1+|x|+|prox_u^\epsilon(x)|)^{l}|x-prox_u^\epsilon(x)|\\&\leq
L(1+|x|+|prox_u^\epsilon(x)|)^{l} |h_{MY,\epsilon}(x)|\epsilon
\\&\leq L(1+|x|+|prox_u^\epsilon(x)|)^{l}  |h(x)|\epsilon\\&\leq L(1+|x|+|prox_u^\epsilon(x)|)^{l}(1+|x|+|x^*|)^l|x-x^*|\epsilon\\&\leq
L(1+|x|+|prox_u^\epsilon(x)|)^{l}(1+|x|+|x^*|)^{l+1}
\epsilon.
\end{aligned}\end{equation}
Since the proximal operator is $1$-Lipschitz and $prox_u^\epsilon(x^*)=x^*$,
we see that \begin{equation}\label{prox-growth}
    |prox_u^\epsilon(x)|\leq |x|^*+|prox_u^\epsilon(x)-prox^{\epsilon}_u(x^*)|\leq |x^*|+|x-x^*|\leq |x|+2\sqrt{R},
\end{equation}
at which point inserting \eqref{prox-growth} into \eqref{eq-grad-dist} yields the result.
\end{proof}

\newtheorem{W2 estimate}[Ass1]{Corollary}
\begin{W2 estimate}\label{W2 estimate}
There exists a constant $\bar{c}>0$ such that \[W_2(\pi_\beta,\pi_\beta^\epsilon)\leq \bar{c} {\epsilon},\] for every $\epsilon>0$.

\end{W2 estimate}

\begin{proof}
Let $x_t$ be the overdamped Langevin SDE given as :
\[dx_t=-h(x_t)dt +\sqrt{2\beta^{-1}}dB_t,\]
and let $y_t$ be given as \[dy_t=-\he(y_t)dt+\sqrt{2\beta^{-1}}dB_t \] for initial conditions $X_0$, $Y_0$ satisfying $\mathcal{L}(X_0)=\mathcal{L}(Y_0)=\mu_\beta$ .
As a result, since $x_t-y_t$ has vanishing diffusion, one can take the time derivative defined a.s ,i.e
\[\begin{aligned}
 \frac{d}{dt}|x_t-y_t|^2 &=-2\langle x_t-y_t,h(x_t)-\he(y_t)\rangle\\&=-2\langle x_t-y_t,h(x_t)-\he(x_t)\rangle -2 \langle x_t-y_t,\he(x_t)-\he(y_t)\rangle \\&\leq
 \frac{m}{2}|x_t-y_t|^2 +\frac{2}{m}|h(x_t)-\he(x_t)|^2 -m |x_t-y_t|^2
 \\& \leq - \frac{m}{2}|x_t-y_t|^2+\frac{2}{m}|h(x_t)-\he(x_t)|^2.
\end{aligned} \]
where the last step was obtained by Young's inequality and the strong convexity of $u_{MY,\epsilon}$.
Since $\mu_\beta$ is the invariant measure for \eqref{eq-grad-dist}, we know that $x_t$ has law equal to $\mu_\beta$ for every $t\geq 0$. So by
\eqref{eq-grad-dist},
\[\E |h(x_t)-h_{MY,\epsilon}(x_t)|^2 \leq \epsilon^2 2^{4l+4}L^2 \E (1+|x_t|+\sqrt{R})^{4l+4}\leq C_{l,L,R} \epsilon^2,\]
where $C_{l,L,R}$ is derived by Proposition \ref{Inv mom}.
As a result,
\[\E \frac{d}{dt}|x_t-y_t|^2\leq -\frac{m}{2}\E |x_t-y_t|^2 +c \epsilon^2, \]
and since the right hand side is finite one can exchange derivative and expectation to obtain
\[\frac{d}{dt} \E |x_t-y_t|^2\leq -\frac{m}{2}\E |x_t-y_t|^2 +c \epsilon^2.\]
Setting $f(t):=\E |x_t-y_t|^2$, by a simple calculations one obtains
\[(e^{\frac{m}{2}t}f(t)-\frac{2c}{m}\epsilon^2 e^{\frac{m}{2}t})'\leq 0, \] which implies by the fundamental theorem of calculus that
\[e^{\frac{m}{2}t}f(t)-\frac{2c}{m}\epsilon^2 e^{\frac{m}{2}t}\leq f(0)-\frac{2c}{m}\epsilon^2=-\frac{2c}{m}\epsilon^2,\]
and thus one concludes that
\begin{equation}\label{eq-xy}
    \E |x_t-y_t|^2\leq \frac{2c}{m}\epsilon^2, \quad  t\geq 0.
\end{equation}
Then as $x_t$ follows the law of the invariant measure, one has the bound
\[W_2(y_t,\mu_\beta)\leq \sqrt{\E |x_t-y_t|^2}\leq \sqrt{\frac{2c}{m}}\epsilon, \]
and as a result,
\[W_2(\mu_\beta,\mu_\beta ^\epsilon)\leq W_2(y_t,\mu_\beta^\epsilon)+ \sqrt{\frac{2c}{m}}\epsilon.\]
Since $y_t$ converges in $W_2$ to its invariant measure, letting $t\rightarrow \infty$ yields the result.
Noticing that $W_2(\pi_\beta,\pi_\beta^\epsilon)\leq W_2(\mu_\beta^\epsilon,\mu_\beta ) $ and setting $\bar{c} = \sqrt{\frac{2c}{m}}$ then completes the proof.
\end{proof}
\section{New Euler-Krylov Polygonal (Tamed) Scheme}\label{sec-tamed}
Our goal is to construct a stable and efficient sampling algorithm with which to obtain approximate samples from $\pi_\beta ^\epsilon$ (and therefore essentially from $\pi_\beta$), given that the log-gradient of the density is of the class described by Assumption \ref{A2}. A natural step would be to use a discretised version of \eqref{eq-underlang} with the Moreau Yosida gradient $h_{MY,\epsilon}$ in place of $\nabla u$. However, such an algorithm would require additional computation in order to calculate the Moreau Yosida gradient at each iteration, dramatically increasing its computational complexity.


Instead we use a new Euler-Krylov polygonal scheme with drift coefficient $\hg$ depending on $\gamma$. This new function has linear growth and additionally satisfies a dissipativity condition that is crucial for proving uniform moment estimates for the algorithm in $\lambda$. In this section we prove these growth and dissipativity properties and additionally demonstrate its convergence to the original gradient.
\newtheorem{tamingdef}[Ass1]{Definition}
\begin{tamingdef}
Let \[f(x)=h(x)-\frac{m}{2}x\]
We define $\fg$ in the following way:
\[\begin{aligned}
\fg(x)&=f(x) \quad &\text{if } & \quad  |f(x)|\leq \sqrt{\gamma}\\
\fg(x)&=\frac{2 f(x)}{1+\gamma^{-\frac{1}{2}}|f(x)|} \quad &\text{if } & \quad |f(x)|>\sqrt{\gamma}.
\end{aligned}\]
Then, \[\hg=\fg+\frac{m}{2}x.\]
\end{tamingdef}
\newtheorem{Propertiesdiss}[Ass1]{Lemma}
\begin{Propertiesdiss}\label{Propertiesdiss}
There holds, \[\langle \hg(x),x\rangle \geq \frac{m}{2}|x|^2-u(0)\]
\end{Propertiesdiss}
\begin{proof}
The proof begins by writing \[\langle \hg(x),x\rangle=\langle \fg(x),x\rangle +\frac{m}{2}|x|^2. \]
We split the proof in two parts depending on the sign of $\langle f(x),x\rangle$.\\\\
If $\langle f(x),x\rangle \geq 0$ then, it is easy to see that
$\langle \fg(x),x\rangle \geq 0$ so
\begin{equation}\label{eq-tamdiss1}
    \langle \hg(x),x\rangle \geq \frac{m}{2}|x|^2.
\end{equation}\\
Alternatively if $ \langle f(x),x\rangle< 0 $ then,
noticing that \[|\langle \fg(x),x\rangle| \leq |\langle f(x),x\rangle|, \] there follows that
\[
    \langle \fg(x),x\rangle \geq -|\langle f(x),x\rangle|=\langle f(x),x\rangle.\]
This leads to
\begin{equation}\label{eq-tamdiss2}
\begin{aligned}
    \langle \hg(x),x\rangle &\geq \langle f(x),x\rangle +\frac{m}{2}|x|^2\\&=\langle h(x),x\rangle -\frac{m}{2}|x|^2+\frac{m}{2}|x|^2\\&=\langle h(x),x\rangle \\&\geq \frac{m}{2}|x|^2-u(0).
\end{aligned}
\end{equation}
Combining \eqref{eq-tamdiss1} and \eqref{eq-tamdiss2} yields
the result.
\end{proof}
\newtheorem{Lemma growth}[Ass1]{Lemma}
\begin{Lemma growth}
There holds \[|\hg(x)|\leq 2\sqrt{\gamma}+ \frac{m}{2}|x| \]
\end{Lemma growth}
\begin{proof}
If $|f(x)|\leq \sqrt{\gamma}$ then, \begin{equation}\label{eq-growth1}
    |\hg(x)|=|f(x) +\frac{m}{2}x|\leq \sqrt{\gamma}+\frac{m}{2}|x|.
\end{equation}
On the other hand, if $|f(x)|>\sqrt{\gamma}$ then, \begin{equation}\label{eq-growth2}
    |\hg(x)|\leq |\fg(x)| +\frac{m}{2}|x|\leq 2\sqrt{\gamma} +\frac{m}{2}|x|
\end{equation}
Combining \eqref{eq-growth1} and \eqref{eq-growth2} yields the result.
\end{proof}
\newtheorem{Lemma 5.3}[Ass1]{Lemma}
\begin{Lemma 5.3}\label{tamed distance general}
Let $p>0$ and a random variable $X$ with finite $(4p+4)(l+1)$ moments. Then,
\[\E |\hg(X)-h(X)|^2\leq c \gamma^{-p}\]
where $c=2^{2p+2} (L+\frac{m}{2})^{2p+2}\sqrt{\E |X|^{(4p+4)(l+1)}} +2^{2p+2} |h(0)|^{2p+2}$
\end{Lemma 5.3}
\begin{proof}
By using the definition of $\hg$
\[\begin{aligned}
    |\hg(X)-h(X)|^2\\&=|\fg(X)-f(X)|^2\\&= |f(X)|^2 \left | \frac{\gamma^{-\frac{1}{2}}|f(X)|-1}{1+\gamma^{-\frac{1}{2}}|f(X)|} \right| \mathds{1}_{\{|f(X)|\geq \sqrt{\gamma}\}}\\&\leq
    |f(X)|^2  \mathds{1}_{\{|f(X)|\geq \sqrt{\gamma}\}}
\end{aligned} \]
Taking expectations and using Cauchy Swartz inequality one obtains,
\[\E |\hg(X)-h(X)|^2\leq \sqrt{\E |f(X)|^4} \sqrt{P(|f(X)|\geq \sqrt{\gamma})}\] and then, an application of Markov's inequality yields
 \[\E|\hg(X)-h(X)|^2\leq \sqrt{\E|f(X)|^4 }\sqrt{\E |f(X)|^{4p}}\gamma^{-p}=\sqrt{\E |f(X)|^{4p+4}} \gamma^{-p}.\]
 Noticing that \[|f(x)|\leq |h(x)|+\frac{m}{2}|x|\leq |h(0)|+L(1+|x|)^{l+1} +\frac{m}{2}|x|\] yields the result.
\end{proof}
\section{Sampling with tKLMC1 Algorithm}\label{se-tKLMC1}
We are now ready to construct the tKLMC1 algorithm by using the gradient term $h_{tam,\gamma}$ as the drift coefficient.
The new algorithm has the initial condition $(\bar{\theta}_0^\lambda,\bar{V_0}^\lambda)=(\theta_0,V_0)$ and is a type of Euler-scheme given by the recursion
\begin{equation}\label{eq-discrete}
    \begin{array}{l}
\bar{V}_{n+1}^{\lambda}=\bar{V}^\lambda_{n}-\lambda\left[\gamma \bar{V}^\lambda_{n}+h_{tam,\gamma}\left(\bar{\theta}^\lambda_{n}\right)\right]+\sqrt{\frac{2 \gamma \lambda}{\beta}} \xi_{n+1}, \\
\bar{\theta}^\lambda_{n+1}=\bar{\theta}^\lambda_{n}+\lambda \bar{V}^\lambda_{n}, \quad n\geq 0,
\end{array}
\end{equation}
where $\lambda>0$ is the step-size and $\{\xi_n\}_{n\geq1}$ are $d$-dimensional independent standard Gaussian random variables. Moreover, it is assumed that $\theta_0$,$V_0$ and $\xi_n$ are independent.

Let us now introduce additional auxiliary processes that play an important role in our analysis. Consider the scaled process $(\zeta_t^{\lambda,n},Z_t^{\lambda,n}):=(\tilde{\theta}_{\lambda t},\tilde{V}_{\lambda t})$ where $(\tilde{\theta}_t,\tilde{V}_t)$ are given in \eqref{eq-underlang}. Additionally, define
\begin{equation}\label{eq-Langevin}
 \begin{aligned}
\mathrm{d} R_{t}^{\lambda} &=-\lambda\left(\gamma R_{t}^{\lambda}+h_{MY,\epsilon}\left(r_{t}^{\lambda}\right)\right) \mathrm{d} t+\sqrt{2 \gamma \lambda \beta^{-1}} \mathrm{~d} B_{t}^{\lambda}, \\
\mathrm{d} r_{t}^{\lambda} &=\lambda R_{t}^{\lambda} \mathrm{d} t,
\end{aligned}
\end{equation}
where $B_t^\lambda:=\frac{1}{\sqrt{\lambda}}B_{\lambda t}$. The new Brownian motion is adapted to its natural filtration $\mathbb{F}^\lambda:=\{\mathcal{F}^\lambda_t\}_{t\geq 0}$ which is independent of $\sigma(\theta_0,V_0)$. In addition, we define the continuous-time interpolation of the algorithm as
\begin{equation}
    \begin{aligned}
 \mathrm{d} {V}_{t}^{\lambda} &=-\lambda\left(\gamma {V}_{\lfloor t\rfloor}^{\lambda}+h_{tam,\gamma} ({\theta^\lambda}_{\lfloor t\rfloor})\right) \mathrm{d} t+\sqrt{2 \gamma \lambda \beta^{-1}} \mathrm{~d} B_{t}^{\lambda}, \\
 \mathrm{d} {\theta}_{t}^{\lambda} &=\lambda {V}_{\lfloor t\rfloor}^{\lambda} \mathrm{d} t, \quad t \geq 0,
 \end{aligned}
\end{equation}
with initial condition $V_0^\lambda=V_0$ and $\theta^\lambda_0=\theta_0$. Note that $\mathcal{L}(\theta^\lambda_{\lfloor t \rfloor},V^\lambda_{\lfloor t \rfloor})=\mathcal{L}(\bar{\theta}^\lambda_{\lfloor t \rfloor},\bar{V}^\lambda_{\lfloor t \rfloor})$.
Finally, we define the underdamped Langevin process $\left(\widehat{\zeta}_{t}^{s, u, v, \lambda}, \widehat{Z}_{t}^{s, u, v, \lambda}\right)$ for $s \leq t$
$$
\begin{aligned}
&\mathrm{d} \widehat{Z}_{t}^{s, u, v, \lambda}=-\lambda\left(\gamma \widehat{Z}_{t}^{s, u, v, \lambda}+h_{MY,\epsilon}\left(\widehat{\zeta}_{t}^{s, u, v, \lambda}\right)\right) \mathrm{d} t+\sqrt{2 \gamma \lambda \beta^{-1}} \mathrm{~d} B_{t}^{\lambda}, \\
&\mathrm{d} \widehat{\zeta}_{t}^{s, u, v, \lambda}=\lambda \widehat{Z}_{t}^{s, u, v, \lambda} \mathrm{d} t,
\end{aligned}
$$
with initial conditions $\widehat{\theta}_{s}^{s, u, v, \lambda}=u$ and $\widehat{V}_{s}^{s, u, v, \lambda}=v$.
\newtheorem{aux process}[Ass1]{Definition}
\begin{aux process} Fix $n \in \mathbb{N}$ and let $T:=\lfloor 1 / \lambda\rfloor$. Then we define
$$
\zeta_{t}^{\lambda, n}:=\widehat{\zeta}_{t}^{n T, \bar{\theta}^\lambda_{n T}, \bar{V}^\lambda_{n T}, \lambda}, \quad \text { and } \quad Z_{t}^{\lambda, n}:=\widehat{Z}_{t}^{n T, \bar{\theta}^\lambda_{n T} , \bar{V}_{n T} ^\lambda, \lambda},
$$
such that the process $\left(\zeta_{t}^{\lambda, n}, Z_{t}^{\lambda, n}\right)_{t \geq n T}$ is an underdamped Langevin process started at time $n T$ with initial conditions $\left(
\theta^\lambda_{n T},
{V}_{n T}^{\lambda}\right)$.
\end{aux process}
Throughout our analysis we assume that
\begin{equation}\label{eq-restrictions1}
    \gamma_{\min,1}=\max\{ \sqrt{\frac{K+m}{\beta}},K,32,\frac{48(2m+1)^2}{m}\} \quad \text{and} \quad \lambda_{\max,1}=\gamma^{-1}_{\min,1}
\end{equation}
\subsection{Moment bounds (tKLMC1)}
In order to proceed with the convergence properties of the algorithm \eqref{eq-discrete} we  first need some high moment estimates. Proofs are postponed to the Appendix.
\newtheorem{second mom}[Ass1]{Lemma}
\begin{second mom}\label{second mom}
For $\lambda<\gamma^{-1}<\frac{m}{48(2m+1)^2}$ one has
\[\sup_n \E |\bar{\theta}^\lambda_n|^2\leq \bar{C}_2,\]
and \[\sup_n \E |\bar{V}^\lambda_n|^2 \leq \bar{B}_{2},\] where the constants $\bar{C}_2$ $\bar{B}_2$ are independent of $\gamma$, have dependence on the dimension at most $\mathcal{O}(\frac{d}{\beta})$ and are given explicitly in the proof.
\end{second mom}
\begin{proof}
Postponed to the Appendix
\end{proof}
\newtheorem{high mom}[Ass1]{Lemma}
\begin{high mom}\label{high mom}
For $\lambda<\gamma^{-1}<\frac{m}{48(2m+1)^2}$
there holds \[\sup_n \E |\bar{\theta}^\lambda_n|^{2q}\leq \bar{C}_{2q},\]
where $\bar{C}_{2q}=\mathcal{O}((\frac{d}{\beta})^{q})$ and is independent of $\gamma$.
\end{high mom}
\begin{proof}
Postponed to the Appendix.
\end{proof}
We conclude this section by presenting the one-step errors for the continuous interpolation of the algorithm.
\newtheorem{One-step}[Ass1]{Lemma}
\begin{One-step}\label{one-step}
For every $t\geq0$ one has the bound \[\E |\theta^\lambda_t-\theta^\lambda_{\lfloor t \rfloor}|^2 \leq \lambda \bar{B}_2, \] and
\[\E |V^\lambda_{\lfloor t \rfloor}-V_t^\lambda|^2\leq \lambda \gamma C_{1,v}, \]
where $C_{1,v}=\max\{\tilde{C}_2,16\frac{d}{\beta}\}$
where $\bar{C}_2$ and $\bar{B}_2$ are given explicitly in the proof of Lemma $\ref{second mom}.$
\end{One-step}
\begin{proof}
Postponed to the appendix.
\end{proof}
\newtheorem{mom-continuous prelim}[Ass1]{Lemma}
\begin{mom-continuous prelim} [\cite{gao2021global} Lemma 16 (i) ]\label{mom-continuous prelim}
Let $U\in\mathcal{C}_1(\mathbb{R}^d)$ be a function satisfying
\begin{enumerate}
\item $U\geq 0$
    \item $\nabla U$ is Lipschitz
    \item There exist $a,b>0$ such that \[\langle x,\nabla U(x)\rangle \geq a|x|^2-b, \quad  x \in \mathbb{R}^d.\]
    \end{enumerate}
    Furthermore, consider the following underdamped Langevin SDE $\{(X_t,Y_t)\}_{t\geq 0}$
    \[\begin{aligned}
     dY_t&=-\gamma Y_t dt -\lambda \nabla U(X_t) dt +\sqrt{\frac{2 \lambda \gamma}{\beta}}dB_t,
     \\ dX_t&=\lambda Y_t dt,\quad t>0.
    \end{aligned}
    \]
    with initial condition $(X_0,V_0)$ such that \[\E |X_0|^2 +\E |Y_0|^2<\infty. \]
    Then one has the bounds, \[\sup_{t\geq 0}\E |X_t|^2<\infty,\] and  \[\sup_{t\geq 0}\E |Y_t|^2<\infty.\]
\end{mom-continuous prelim}
\newtheorem{mom-continuous}[Ass1]{Corollary}
\begin{mom-continuous}\label{mom-continuous}
Let Assumptions \ref{A1},\ref{A3} hold. Then for $0<\lambda<\lambda_{max}$ one has
\[\sup_n \sup_{t\geq nT} \E |\zeta^{\lambda,n}_t|^2<\infty,\]
and
\[\sup_n \sup_{t\geq nT} \E |Z^{\lambda,n}_t|^2<\infty.\]
\end{mom-continuous}
\begin{proof}
One notices that due to \eqref{eq-comp0}, \eqref{eq-dissipative}, \eqref{eq-Hess} the function $u_{MY,\epsilon}$ satisfies the assumptions of Lemma \ref{mom-continuous prelim}, the process $(\zflt,Z^{\lambda,n}_t)$ is the solution to an underdamped Langevin SDE with initial condition $(\theta^\lambda_{nT}, V^\lambda_{nT})$, and so by Lemma \ref{second mom} there holds that
\[\E|\theta^\lambda_{nT}|^2+\E|V^\lambda_{nT}|^2\leq  \sup_n\E|\bar{\theta}^\lambda_{n}|^2+\sup_n\E|\bar{V}^\lambda_{n}|^2<\infty. \]
Applying the result of Lemma \ref{mom-continuous prelim} concludes the proof.
\end{proof}
\subsection{Convergence of the tKLMC1 algorithm}
Before we proceed with convergence results, we present the fundamental contraction property that is needed for the derivation of the main results.
\newtheorem{Contraction}[Ass1]{Theorem}
\begin{Contraction} [ cf\cite{Dalaylan}, Theorem 1]\label{Contraction}
Let $\epsilon>0$ be such that $u_{MY,\epsilon} \in \mathcal{C}^2$ and $m I_d\leq Hess(u_{MY,\epsilon})\leq K I_d$. Let $\gamma\geq \sqrt{\frac{m+K}{\beta}}$. Then, if $Rr^\lambda_t$ and ${Rr^\lambda}'_t$ are two solutions of the Langevin SDE \eqref{eq-Langevin} with initial condition $Rr_0=(R_0,r_0)$ and $Rr_0'=(R'_0,r'_0)$ respectively, one has the bound
\[W_2(\mathcal{L}(Rr^\lambda_t) ,\mathcal{L}({Rr^\lambda}'_t))\leq \sqrt{2}\exp{(-\frac{\lambda m}{\beta\gamma}t)}W_2(\mathcal{L}(Rr_0),\mathcal{L}({Rr'_0})). \]
\end{Contraction}
The following corollary follows immediately from the previous result.
\newtheorem{contrinv}[Ass1]{Corollary}
\begin{contrinv}\label{contrinv}
Let \ref{A1}-\ref{A3} hold. Let $(r^\lambda_t,R^\lambda_t)$ be the underdamped Langevin SDE as in \eqref{eq-Langevin} with initial condition $(\theta_0,V_0)$. There holds
\[W_2(\mathcal{L}(R_t^\lambda,r_t^\lambda),\pi_\beta^\epsilon)\leq \sqrt{2}\exp{(-\frac{\lambda m}{\beta\gamma}t)} W_2(\mathcal{L}(\theta_0,V_0),\pi_\beta^\epsilon). \]
\end{contrinv}
In order to prove the main convergence result, one uses the following splitting:
\begin{align}\label{eq-triangle}
    W_n((\mathcal{L}(\bar{V}_t,{\bar{\theta}_t}),\pi_\beta) \leq W_2(\mathcal{L}&(\bar{V}_t,{\bar{\theta}_t}),\mathcal{L}(Z_t^{\lambda,n},\zeta_t^{\lambda,n}))+W_2(\mathcal{L}(Z_t^{\lambda,n},\zeta_t^{\lambda,n}),\mathcal{L}(R_t^\lambda,r_t^\lambda))\nonumber \\
    &+W_2(\mathcal{L}(R_t^\lambda,r_t^\lambda),\pi_\beta^\epsilon)+W_2(\pi_\beta^\epsilon,\pi_\beta).
\end{align}
In order to bound $W_2(\mathcal{L}(\bar{V}_t,{\bar{\theta}_t}),\mathcal{L}(Z_t^{\lambda,n},\zeta_t^{\lambda,n}))$ the following Lemma is required, which demonstrates the good approximating properties of the tamed scheme. From now on we assume that $\gamma^{-1}\leq \min\{\epsilon,\frac{1}{32},\frac{m}{48(2m+1)^2}\}$ and $\lambda<\gamma^{-1}.$
\newtheorem{tamed distance}[Ass1]{Lemma}
\begin{tamed distance}\label{tamed distance}
Let Assumptions \ref{A1}-\ref{A3} hold and $q>0$.
Then, by Lemma \ref{tamed distance general} where $p=2q+1$ and Lemma \ref{high mom}, there follows
\[\E|h_{tam,\gamma}(\tfls)-h(\tfls)|^2\leq \frac{C_A}{16} \gamma^{-(2q+1)},\]
where $C_A:=\mathcal{O}(2^{(2(l+2)+2q(2l+2))}\bar{C}_{2(l+2)+2q(2l+2)})$.
\end{tamed distance}
\newtheorem{Aux-alg distance}[Ass1]{Proposition}
\begin{Aux-alg distance}\label{Aux-alg distance}
Let Assumptions \ref{A1}-\ref{A3} hold.
Then, for every $n\in \mathbb{N}$ and for $nT\leq t\leq (n+1)T$ one has
\[W_2\left(\mathcal{L}(V_t^\lambda,\theta^\lambda_t),\mathcal{L}(Z_t^{\lambda,n},\zeta_t^{\lambda,n})\right)\leq C \left(\sqrt{\lambda}\sqrt{\gamma}+\sqrt{\frac{\slg}{\gamma}}\right),\]
where $C:=\sqrt{6(1+C_{1,v})}$ where $C_{1,v}$ is given in Lemma \ref{one-step} and $\slg$ is given by \eqref{eq-slg}.
\end{Aux-alg distance}
 \newtheorem{Contr}[Ass1]{Theorem}
 \begin{Contr}\label{Contr}
 Let Assumptions \ref{A1}-\ref{A3} hold and let $(r^\lambda_t,R^\lambda_t)$ be the solution to the underdamped Langevin SDE \eqref{eq-Langevin} with initial condition $(\theta_0,V_0)$. Then for every $n\in \mathbb{N}$, $nT\leq t\leq (n+1)T$ \[W_2(\mathcal{L}(\zeta^{\lambda,n}_t,Z^{\lambda,n}_t),\mathcal{L}(r^{\lambda}_t,R^\lambda_t))\leq C'\left(\sqrt{\lambda}\gamma^\frac{3}{2} +\sqrt{\gamma} \sqrt{\slg}\right)\]
 where $C'$ is given explicitly in the proof and  $\slg$ is given by \eqref{eq-slg}.
 \end{Contr}
\section{Sampling with the tKLMC2 algorithm}\label{se-tKLMC2}
In this section we propose a tamed version of KLMC2, an algorithm which was first developed in \cite{cheng2018underdamped} and analysed in depth in \cite{Dalaylan} and \cite{gao2021global} under the assumption of Lipschitz continuity for the gradient.
The creation of this algorithm is motivated by the fact that  \eqref{eq-underlang}, after applying It\^{o}'s formula to the product $e^{\gamma t}V_t$ and by following standard calculations, can be rewritten as
\begin{equation}\label{eq-underlangtwo}
\begin{aligned}
    \Bar{V}_t&=e^{-\gamma t}V_0 - \int_0^t e^{-\gamma(t-s)}\nabla u(\Bar{\theta}_s)ds+\sqrt{\frac{2\gamma}{\beta}} \int_0^t e^{-\gamma(t-s)} dB_s\\
    \Bar{\theta_t}&=\theta_0+\int_0^t \Bar{V}_s ds.
\end{aligned}
\end{equation}
Since the class of functions satisfying Assumption \ref{A2} allows for gradients with superlinear growth, we shall tame the gradient part of the drift cofficient using the tamed technique developed in Section \ref{sec-tamed}. Our strategy to prove the moment bounds and convegence rates is similar to the one used for the proof of the respective results for the tKLMC1 algorithm in Section \ref{se-tKLMC1}.

The new iterative scheme of tKLMC2 is given by
\begin{equation}\label{eq-tKLMC2}
    \begin{aligned}
    \bar{Y}^\lambda_{n+1}&=  \psi_0(\lambda) \yn-\psi_1(\lambda) \hg(\xn) +\sqrt{2\gamma \beta^{-1}} \Xi_{n+1}\\
    \bar{x}^\lambda _{n+1}&=\xn+\psi_1(\lambda)\yn-\psi_2(\lambda)\hg(\xn)+\sqrt{2\gamma \beta^{-1}} \Xi'_{n+1}
    \end{aligned}
\end{equation}
where $\psi_0(t)=e^{-\gamma t}$ and $\psi_{i+1}=\int_0^t \psi_i(s)ds$
where $\left(\boldsymbol{\Xi}_{k+1}, \boldsymbol{\Xi}_{k+1}^{\prime}\right)$ is a $2 d $-dimensional centered Gaussian vector satisfying the following conditions:
\begin{itemize}
    \item
- $\left(\boldsymbol{\Xi}_{j}, \boldsymbol{\Xi}_{j}^{\prime}\right)^{\prime}$ s are iid and independent of the initial condition $\left({V}_{0}, \theta_{0}\right)$,
\item for any fixed $j$, the random vectors $\left(\left(\boldsymbol{\Xi}_{j}\right)_{1},\left(\boldsymbol{\Xi}_{j}^{\prime}\right)_{1}\right),\left(\left(\boldsymbol{\Xi}_{j}\right)_{2},\left(\boldsymbol{\Xi}_{j}^{\prime}\right)_{2}\right), \ldots,\left(\left(\boldsymbol{\Xi}_{j}\right)_{d},\left(\boldsymbol{\Xi}_{j}^{\prime}\right)_{d}\right)$ are iid with the covariance matrix
$$
\mathbf{C}=\int_{0}^{\lambda}\left[\psi_{0}(t), \psi_{1}(t)\right]^{\top}\left[\psi_{0}(t) , \psi_{1}(t)\right] d t.
$$
\end{itemize}

\noindent At this point and in view of \eqref{eq-underlangtwo}, one claims that the continuous time interpolation of \eqref{eq-tKLMC2} is given by
\begin{equation}\label{eq-contint}
    \begin{aligned}
    \Tilde{Y}_t&=e^{- \gamma(t-n\lambda)}\Tilde{Y}_{n\lambda}-\int_{n\lambda}^t e^{-\gamma(t-s)} \hg(\tilde{x}_{n\lambda}) ds +\sqrt{2\gamma\beta^{-1}}\int_{n\lambda}^t e^{-\gamma(t-s)}dW_{s}
    \\\Tilde{x}_t&=\tilde{x}_{n\lambda} +\int_{n\lambda}^t \tilde{Y}_{s} ds,
    \end{aligned}
\end{equation}
which naturally leads to the following Lemma.

\newtheorem{remarkgridpoints}[Ass1]{Lemma}
\begin{remarkgridpoints}\label{remarkgridpoints}
Let $\lambda >0$, $(\yn,\xn)$ be given by \eqref{eq-tKLMC2} and $(\tilde{Y}_{n\lambda},\tilde{x}_{n\lambda})$ be given by \eqref{eq-contint}. Then,
 \[\mathcal{L}(\yn,\xn)=\mathcal{L}(\tilde{Y}_{n\lambda},\tilde{x}_{n\lambda}), \quad \forall n \in \mathbb{N}.\]
 \end{remarkgridpoints}
 
 The details of the proof can be found in the proof section.

\newtheorem{auxp}[Ass1]{Definition}
\begin{auxp}
For every $t\in [n\lambda, (n+1)\lambda]$ we define the auxiliary process such that $(Q_{n \lambda},p_{n \lambda})=(\tilde{Y}_{n\lambda},\tilde{x}_{n\lambda})$
\begin{equation}\label{eq-auxp}
\begin{aligned}
dQ^{\lambda,n}_t&=-\gamma Q^{\lambda,n}_t-\he(p^{\lambda,n}_t)+\sqrt{2\gamma \beta^{-1}} dW_t
\\ dp^{\lambda,n}_t&=Q^{\lambda,n}_t
\end{aligned}
\end{equation}
or alternatively
\[\begin{aligned}
Q^{\lambda,n}_t&=e^{-\gamma (t-j)\lambda)}Q^{\lambda,n}_{ n\lambda}-\int_{n\lambda}^t e^{-\gamma(t-s)}\he(p^{\lambda,n}_s) ds+\sqrt{2\gamma\beta^{-1}}\int_{n\lambda}^t e^{-\gamma(t-s)}dWs
\\p^{\lambda,n}_t&=p^{\lambda,n}_{n\lambda}+\int_{n\lambda}^t Q^{\lambda,n}_s ds .
\end{aligned}
\]
where $W_t$ is the same Brownian motion as in \eqref{eq-contint}.
Throughout our analysis we assume that
\begin{equation}\label{eq-restrictions2}
\gamma_{\min}=\sqrt{\frac{2K}{\beta}} \quad \text{and} \quad    \lambda_{\max}=\frac{1}{C'6}\gamma^{-5} .
\end{equation}
where $C'_6$ is given in the proof section.
\end{auxp}
\subsection{Moment bounds(tKLMC2)}
\newtheorem{mom-KLMC2}[Ass1]{Lemma}
\begin{mom-KLMC2}\label{mom-KLMC2}
Under Assumptions \ref{A1}-\ref{A3}, $\lambda<\lambda_{\max,2}$, $\gamma>\gamma_{\min,2}$ there holds,
\[\E |\xn|^2\leq C_x\]
and \[\E|\yn|^2\leq C_y \gamma^2. \]
where $C_x,C_y$ are given in the proof and depend on the $d,\beta$ (their dependence is at most $\mathcal{O}(\frac{d}{\beta})$) and $\E|V_0|^2+\E |\theta_0|^2$ and are also independent of $\gamma$.
\end{mom-KLMC2}
\newtheorem{HighmomKLMC2}[Ass1]{Lemma}
\begin{HighmomKLMC2}\label{HighmomKLMC2}
Under Assumptions \ref{A1}-\ref{A3}, $\lambda<\lambda_{\max,2}$, $\gamma>\gamma_{\min,2}$ there holds
\[\E|\xn|^{2q}\leq \sqrt{C_{4q}}.\]
\end{HighmomKLMC2}
\subsection{Convergence of the KLMC2 algorithm}
\newtheorem{tKLMC2conv}[Ass1]{Lemma}
\begin{tKLMC2conv}\label{tKLMC2conv}
Let $n\in \mathbb{N}$. Under Assumptions \ref{A1}-\ref{A3}, $\lambda<\lambda_{\max,2}$, $\gamma>\gamma_{\min,2}$ for any $t \in [n\lambda,(n+1)\lambda]$, $n\in \mathbb{N}$ there holds
\[W_2\left(\mathcal{L}(p^{\lambda,n}_t,Q^{\lambda,n}_t),\mathcal{L}(\bar{x}^\lambda_t,\bar{Y}^\lambda_t)\right)\leq 2 R_{\lambda,\gamma}\]
where $R_{\lambda,\gamma}$ is given by \eqref{eq-Rlg}.
\end{tKLMC2conv}
\newtheorem{Lemmacontra}[Ass1]{Lemma}
\begin{Lemmacontra}\label{Lemmacontra}
Let $(X_t,Y_t)$ the kinetic Langevin SDE with initial condition $(\theta_0,V_0).$ Let $n\in \mathbb{N}$. Under Assumptions \ref{A1}-\ref{A3}, $\lambda<\lambda_{\max,2}$, $\gamma>\gamma_{\min,2}$ for any $t \in [n\lambda,(n+1)\lambda]$, there holds,
\[W_2(\mathcal{L}(p_t^{\lambda,n},Q_t^{\lambda,n}),\mathcal{L}(X_t,Y_t))\leq C' \frac{\gamma}{\lambda} \sqrt{R_{\lambda,\gamma}}\]
where $R_{\lambda,\gamma}$ is given by \eqref{eq-Rlg}.
\end{Lemmacontra}
\section{Solving the optimization problem}\label{se-opt}
The main goal of this section is to build a process $x_n$ associated with our algorithm that can be used to minimize the excess risk optimization problem $u(x_n)-u(x^*)$.

To this end, let $R_0=\sqrt{\frac{2u(0)}{m}} + \sqrt{\frac{(u(0)+1)d}{\beta}} +1$, $E=\bar{B}(0,R_0)$ and $E'=\bar{B}(0,R_0+1)$, and let $X_n$, $X$ be random variables satisfying $\mathcal{L}(X_n)=\mathcal{L}(\bar{\theta}^\lambda_n)$, $\mathcal{L}(X)=\mu_\beta$ and  \[\sqrt{\E |X_n-X|^2}=W_2(\mathcal{L}(\bar{\theta}^\lambda_n),\mu_\beta).\]
Then we define $x_n:=\bar{\theta}_n^\lambda\mathds{1}_E(\bar{\theta}_n^\lambda).$
\newtheorem{opt-lemma1}[Ass1]{Lemma}
\begin{opt-lemma1}\label{opt-lemma1}
Let Assumptions \ref{A1}-\ref{A3} hold.
Then,
\[\E u(x_n)-\E u(X\mathds{1}_{E}(X_n))\leq L(1+2R_0)^{l+1}W_2(\mathcal{L}(\bar{\theta}_n^\lambda),\mu_\beta).\]
\end{opt-lemma1}
\newtheorem{opt-lemma2}[Ass1]{Lemma}
\begin{opt-lemma2}\label{opt-lemma2}
Under Assumptions \ref{A1}-\ref{A3} there holds
\[\E u(X\mathds{1}_E(X_n))-u(x^*)\leq C''((E|X|^{2l+4})^{1/2}+1)W_2(\bar{\theta}_n^\lambda,\mu_\beta)+C\E |X-x^*|^2 +2u(0)W_2^2(\bar{\theta}_n^\lambda,\mu_\beta) ,\]
where $E|X|^{2l+4}$ can be controlled using Proposition \ref{inv-measure bound}.
\end{opt-lemma2}
\newtheorem{opt-lemma3}[Ass1]{Lemma}
\begin{opt-lemma3}\label{opt-lemma3}
Under Assumptions \ref{A1},\ref{A2} one has
\[\E |X-x^*|^2\leq \frac{2d}{m\beta}.\]
\end{opt-lemma3}
\newtheorem{opt-solution}[Ass1]{Corollary}
\begin{opt-solution}\label{opt-solution}
Under Assumptions \ref{A1},\ref{A2} there holds
\[\E u\left(\bar{\theta}_n^\lambda\mathds{1}_E(\bar{\theta}_n^\lambda)\right)-u(x^*)\leq C'W_2(\mathcal{L}(\bar{\theta}^\lambda_n,\bar{V}_n^\lambda),\pi_\beta) +2u(0)W_2(\mathcal{L}(\bar{\theta}^\lambda_n,\bar{V}_n^\lambda),\pi_\beta)+\frac{2d}{m\beta}, \]
where the Wasserstein distance can be bounded by Theorem \ref{Final rate}.
\end{opt-solution}

\newpage

\appendix
\section{Proofs of section \ref{se-tKLMC1}}
For the rest of the section
\begin{equation*}
    M_n=\frac{\gamma^2}{4} |\bar{\theta}^\lambda_n+\gamma^{-1}\bar{V}^\lambda_n|^2 +\frac{1}{4} |\bar{V}^\lambda_n|^2-\frac{r\gamma^2}{4} |\bar{\theta}^\lambda_n|^2.
\end{equation*}
\begin{proof}[\textbf{Proof of Lemma} \ref{second mom} ]
Let
\begin{equation*}
    \Delta_n=\bar{\theta}^\lambda_n+\gamma^{-1} \bar{V}^\lambda_n -\frac{\lambda}{\gamma}h_{tam,\gamma}(\bar{\theta}^\lambda_n),
\end{equation*}
and \begin{equation*}
    E_n=\bar{V}^\lambda_n-\lambda \gamma \bar{V}^\lambda_n-\lambda h_{tam,\gamma}(\bar{\theta}^\lambda_n).
\end{equation*}
First observe that
\[|\bar{\theta}^\lambda_{n+1} +\gamma^{-1}\bar{V}^\lambda_{n+1}|^2=|\Delta_n|^2 +2\sqrt{\frac{2\lambda}{\gamma\beta}}\langle \Delta_n,\xi_{n+1}\rangle + \frac{2\lambda}{\gamma \beta}|\xi_{n+1}|^2.\]
Furthermore, there holds
\begin{equation*}
    \begin{aligned}
    |\Delta_n|^2&=|\bar{\theta}^\lambda_n+\gamma^{-1}\bar{V}^\lambda_n|^2 -2\lambda \gamma^{-1} \langle \bar{\theta}^\lambda_n+\gamma^{-1}\bar{V}^\lambda_n,h_{tam,\gamma}(\bar{\theta}^\lambda_n)\rangle + \lambda^2\gamma^{-2}|h_{tam,\gamma}(\bar{\theta}^\lambda_n)|^2
    \\&=|\tn +\gamma^{-1} \vn|^2 -2\lambda\gamma^{-2}\langle \vn,\hg(\tn)\rangle \\&-2\lambda \gamma^{-1}\langle \tn,\hg(\tn)\rangle +\lambda^2 \gamma^{-2} |\hg(\tn)|^2
    \\&\leq |\tn +\gamma^{-1} \vn|^2 +2\lambda\gamma^{-2}| \vn||\hg(\tn)| \\&-2\lambda \gamma^{-1}\langle \tn,\hg(\tn)\rangle +\lambda^2 \gamma^{-2} |\hg(\tn)|^2
    \\&\leq |\tn +\gamma^{-1}\vn|^2 + \lambda \gamma^{-2}\frac{2}{m} |\vn|^2 + \lambda\gamma^{-2} \frac{m}{2} |\hg(\tn)|^2 -2\lambda \gamma^{-1}A|\tn|^2\\&+ 2\lambda \gamma^{-1}B +\lambda^2\gamma^{-2} |\hg(\tn)|^2
    \end{aligned}
    \end{equation*}
    where in the last step we used the dissipativity property of $\hg$ given in Lemma \ref{Propertiesdiss}.
    This leads to
    \begin{equation}\label{eq-mom1}
    \begin{aligned}
    \frac{\gamma^2}{4}    \left(|\bar{\theta}_{n+1}+\gamma^{-1}V_{n+1}|^2-|\tn+\gamma^{-1}\vn|^2\right) &\leq  \lambda \frac{8}{m} |\vn|^2 + \lambda \frac{m}{8} |\hg(\tn)|^2 -\lambda\frac{m}{2} \gamma|\tn|^2\\&+ \lambda \gamma \frac{B}{2} +\lambda^2\frac{1}{4} |\hg(\tn)|^2 \\&+ \gamma^2\sqrt{\frac{\lambda}{2\gamma\beta}}\langle \Delta_n,\xi_{n+1}\rangle + \frac{\lambda\gamma}{2 \beta}|\xi_{n+1}|^2.
    \end{aligned}
    \end{equation}
    \begin{equation}\label{eq-mom2}
     \begin{aligned}
    \frac{1}{4}(|\bar{V}^\lambda_{n+1}|^2-|\bar{V}^\lambda_n|^2)&\leq -\frac{\lambda \gamma }{2}|\bar{V}^\lambda_n|^2 +\frac{\lambda ^2\gamma^2}{4}|\bar{V}^\lambda_n|^2 +\frac{\lambda}{2}|\langle \sqrt{\gamma}\bar{V}^\lambda_n,\sqrt{\gamma^{-1}}h_{tam,\gamma}(\bar{\theta}^\lambda_n)\rangle|+\frac{\lambda^2}{4}|h_{tam,\gamma}(\bar{\theta}^\lambda_n)|^2\\&
   + \sqrt{\frac{\lambda\gamma}{2\beta}}\langle E_n,\xi_{n+1}\rangle +\frac{\lambda \gamma }{2\beta} |\xi_{n+1}|^2
   \\&\leq -\frac{\lambda \gamma }{2}|\bar{V}^\lambda_n|^2 +\frac{\lambda ^2\gamma^2}{4}|\bar{V}^\lambda_n|^2 +\frac{1}{4}\lambda \gamma |\vn|^2 + \frac{1}{4}\lambda \gamma^{-1} |\hg(\tn)|^2   \\&+\frac{\lambda^2}{4}|h_{tam,\gamma}(\bar{\theta}^\lambda_n)|^2
   + \sqrt{\frac{\lambda\gamma}{2\beta}}\langle E_n,\xi_{n+1}\rangle +\frac{\lambda \gamma }{2\beta} |\xi_{n+1}|^2.
   \end{aligned}
   \end{equation}
   Finally,
\[-\frac{\gamma^2 r}{4} ( |\bar{\theta}^\lambda_{n+1}|^2- |\bar{\theta}^\lambda_n|^2)=-\lambda \frac{\gamma^2 r}{2}\langle \bar{\theta}^\lambda_n,\bar{V}^\lambda_n\rangle -\lambda^2 \gamma^2 \frac{r}{4}  |\bar{V}^\lambda_n|^2. \]
Observing that
\[-\frac{\gamma}{2} \langle \bar{\theta}^\lambda_n,\bar{V}^\lambda_n\rangle \leq -M_n +\frac{\gamma^2}{4}  |\bar{\theta}^\lambda_n|^2 +\frac{1}{2}|\bar{V}^\lambda_n|^2,\]
yields
\begin{equation}\label{eq-mom3}
    -\frac{\gamma^2 r}{4} ( |\bar{\theta}^\lambda_{n+1}|^2- |\bar{\theta}^\lambda_n|^2)\leq -\lambda r \gamma M_n + \lambda r\gamma^3 |\bar{\theta}^\lambda_n|^2 + \lambda r \gamma \frac{1}{2}  |\bar{V}^\lambda_n|^2.
\end{equation}
Adding \eqref{eq-mom1}, \eqref{eq-mom2} and \eqref{eq-mom3} yields
\begin{equation}\label{eq-corecontraction}
    \begin{aligned}
    M_{n+1}&\leq (1-\lambda r \gamma)M_n + \lambda \left(\lambda\gamma^2\frac{1}{4} +\frac{8}{m}+\frac{1}{4}+\frac{r}{2}\gamma -\frac{\gamma}{2}\right) |\vn|^2\\& + \lambda \left( -\frac{m\gamma}{4} +r\gamma^3\right)|\tn|^2\\&+
    \lambda \left( \frac{\lambda}{2} +\frac{m}{8}+\frac{\lambda}{2}\right)|\hg(\tn)|^2\\&+\gamma^2\sqrt{\frac{\lambda}{2\gamma\beta}}\langle \Delta_n,\xi_{n+1}\rangle +\sqrt{\frac{\lambda\gamma}{2\beta}}\langle E_n,\xi_{n+1}\rangle +3\frac{\lambda \gamma }{\beta} |\xi_{n+1}|^2 +\lambda \gamma \frac{B}{2}
    \\&\leq (1-\lambda r\gamma) M_n + \lambda \left(\lambda\gamma^2\frac{1}{4} +\frac{8}{m}+\frac{1}{4}+\frac{r}{2}\gamma -\frac{\gamma}{2}\right) |\vn|^2
    \\& +\lambda \left( -\frac{m\gamma}{4}+r\gamma^3 +2\lambda \frac{m^2}{4} +\frac{m^3}{8}\right)|\tn|^2
    \\&+ \gamma^2\sqrt{\frac{\lambda}{2\gamma\beta}}\langle \Delta_n,\xi_{n+1}\rangle +\sqrt{\frac{\lambda\gamma}{2\beta}}\langle E_n,\xi_{n+1}\rangle +3\frac{\lambda \gamma }{\beta} |\xi_{n+1}|^2 +\lambda \gamma \frac{B}{2} +4\lambda^2 \gamma +\frac{m}{2}\lambda \gamma
    \\&\leq (1-\lambda \gamma r) M_n + K_n.
\end{aligned}
\end{equation}
where \begin{equation}\label{eq-Kn}
    K_n= \gamma^2\sqrt{\frac{\lambda}{2\gamma\beta}}\langle \Delta_n,\xi_{n+1}\rangle +\sqrt{\frac{\lambda\gamma}{2\beta}}\langle E_n,\xi_{n+1}\rangle +3\frac{\lambda \gamma }{\beta} |\xi_{n+1}|^2 +\lambda \gamma \frac{B}{2} +4\lambda^2 \gamma +\frac{m}{2}\lambda \gamma
    \end{equation}
    Setting $r=\frac{m}{8}\gamma^{-2}$ for $\gamma> \frac{8}{m}+\frac{1}{4}+\frac{m}{8}\gamma^{-1}+\lambda \gamma^2$ and $\gamma>\frac{8}{m}(2\lambda\frac{m^2}{4} +\frac{m^3}{16})$ one obtains
    \begin{equation}
        \E [M_{n+1}|M_n]\leq (1-\lambda \frac{m}{8}\gamma^{-1}) M_n +\lambda \gamma (\frac{3}{\beta}d +4\lambda +\frac{u(0)}{2}+\frac{m}{2})
    \end{equation}
    As a result,
    \[\E M_{n+1}\leq (1-\lambda \gamma^{-1}\frac{m}{8})^n \E M(0) +\gamma^2 \frac{8}{m}\left(\frac{3}{\beta}d +4 +\frac{u(0)}{2}+\frac{m}{2}\right). \]
    Since \begin{equation}\label{eq-inequal}
    M_n\geq \max \left\{\frac{1}{8}(1-2 r)  \gamma^{2}|\bar{\theta}^\lambda_n|^{2}, \frac{1}{4}(1-2 r)|\bar{V}^\lambda_n|^{2}\right\},
\end{equation}
and \[\E M(0)=\mathcal{O}(\gamma^2),\] then,
\begin{equation}\label{eq-bound theta}
 \sup_n \E |\bar{\theta}^\lambda_n|^2\leq \bar{C}_2,
\end{equation}
where $\bar{C}_2=4\E |\theta_0|^2 + 6\E |V_0|^2 +\frac{64}{m}\left(\frac{3}{\beta}d +4 +\frac{u(0)}{2}+\frac{m}{2}\right). $
We conclude the proof by finding a better bound for the moments of $\bar{V}^\lambda_n$. There holds
\[\begin{aligned}
 \E |\bar{V}^\lambda_{n+1}|^2 &\leq (1-\lambda \gamma)\E |\bar{V}^\lambda_n|^2 -2(1-\lambda\gamma)\lambda\E  \langle \bar{V}^\lambda_n,h_{tam,\gamma}(\bar{\theta}^\lambda_n)\rangle + \lambda^2 \E |h_{tam,\gamma}(\bar{\theta}^\lambda_n)|^2 + \frac{2\lambda \gamma}{\beta }d\\&\leq (1-\lambda \gamma)\E |\bar{V}^\lambda_n|^2 + \frac{1}{2}\lambda \gamma \E |\bar{V}^\lambda_n|^2 + 2\lambda \gamma^{-1} \E |h_{tam,\gamma}(\bar{\theta}^\lambda_n)|^2 +\lambda^2 \E |h_{tam,\gamma}(\bar{\theta}^\lambda_n)|^2 +\frac{2\lambda\gamma}{\beta }d\\&
 \leq (1-\frac{\lambda\gamma}{2})\E |\bar{V}^\lambda_n|^2 +
 4\lambda \frac{m^2}{4} \bar{C}_2  +8 +\lambda^2 (m^2 +4\gamma \bar{C}_2)+\frac{2\lambda\gamma}{\beta }d .
\end{aligned}\]
where the last step is derived from the growth of $\hg$ and \eqref{eq-bound theta}.
After $n$ iterations one obtains,
\begin{equation}
    \E |\bar{V}^\lambda_{n+1}|^2 \leq (1-\lambda \frac{\gamma}{2})^n\E |V_0|^2 + 2m^2\bar{C}_2 +16 +4m^2 +8\bar{C}_2 +\frac{4}{\beta}d.
\end{equation}
so one concludes that
\[\sup_n \E |\vn|^2\leq \bar{B}_2\]
where $\bar{B}_2=\E |V_0|^2 + 2m^2\bar{C}_2 +16 +4m^2 +8\bar{C}_2 +\frac{4}{\beta}d.$
\end{proof}

\begin{proof}[\textbf{Proof of Lemma} \ref{high mom}]
Let \begin{equation*}
    f_n=\gamma^2\sqrt{\frac{\lambda}{2\gamma\beta}}\langle \Delta_n,\xi_{n+1}\rangle +\sqrt{\frac{\lambda\gamma}{2\beta}}\langle E_n,\xi_{n+1}\rangle,
\end{equation*}
and set $g_n:=K_n-f_n$ where $K_n$ is given by \eqref{eq-Kn} and $s=1-\frac{\lambda m}{8\gamma}$. We also define $\E_n(\cdot)=\E [ \cdot|M_n]$. Taking conditional expectations in \eqref{eq-corecontraction} yields
\begin{equation}\label{eq-base}
    \begin{aligned}
    \E_n (M_{n+1}^{2q})&\leq |s M_n|^{2q} +2q |sM_n|^{2(q-1)}\E_n ( s M_nK_n )+
    \sum_{i=2}^{2q} C^i_{2q}\E_n [|s\bar{V}^\lambda_n|^{2q-i}|K_n|^i]
    \\&\leq |s M_n|^{2q} +2q |sM_n|^{2q-1}  \E g_n+
   \ \sum_{i=0}^{2 q-2}\left(\begin{array}{c}
2 q \\
i+2
\end{array}\right)\E_n\left [|sM_n|^{2q-2-i}|K_n|^l|K_n|^2\right]\\&=
     |s M_n|^{2q} +2q |sM_n|^{2q-1}  \lambda \gamma C_2+
    (\begin{array}{c}
2 q \\
2
\end{array}) \sum_{i=0}^{2 q-2}(\begin{array}{c}
2 q-2 \\
i
\end{array}) C^i_{2q-2}\E_n\left [|sM_n|^{2q-2-i}|K_n|^{l+2}\right]\\&\leq  |s M_n|^{2q} +2q |sM_n|^{2q-1}  \lambda \gamma C_2+q(2q-1)\E_n\left[(|sM_n|+|K_n|)^{2q-2}|K_n|^2\right]
    \\&\leq  s |M_n|^{2q} +2q |M_n|^{2q-1}  \lambda \gamma C_2\\&+q(2q-1)2^{2q-3}|M_n|^{2q-2}\E_n |K_n|^2+q(2q-1)2^{2q-3}\E_n |K_n|^{2q}.
    \end{aligned}
\end{equation}
We proceed by bounding the moments of $K_n$. Firstly we calculate that
\begin{equation*}
    \begin{aligned}
    |f_n|^2&\leq 2\gamma^3 \beta^{-1}\lambda  |\Delta_n|^2|\xi_{n+1}|^2 +\lambda \gamma \beta^{-1}|E_n|^2|\xi_{n+1}|^2\\&\leq
    2\gamma^3 \beta^{-1}\lambda |\xi_{n+1}|^2|\bar{\theta}^\lambda_n+\gamma^{-1}\bar{V}^\lambda_n -\lambda\gamma^{-1}h_{tam,\gamma}(\bar{\theta}^\lambda_n)|^2)
    \\&+ \lambda \gamma \beta^{-1}|\xi_{n+1}|^2|(1-\lambda\gamma)\bar{V}^\lambda_n-\lambda h_{tam,\gamma}(\bar{\theta}^\lambda_n)|^2
    \\&\leq 6\gamma^3 \beta^{-1}\lambda |\xi_{n+1}|^2
    \left(|\bar{\theta}^\lambda_n|^2+\gamma^{-2}|\bar{V}^\lambda_n|^2+\frac{1}{2}\lambda^2\gamma^{-1}(m^2|\bar{\theta}^\lambda_n|^2+4\gamma)\right)\\&+
    \lambda \gamma \beta^{-1}|\xi_{n+1}|^2\left( (1-\lambda\gamma)^2|\bar{V}^\lambda_n|^2 +\frac{1}{2}\lambda^2  (m^2|\bar{\theta}^\lambda_n|^2+4\gamma)\right)
    \\&\leq C_3 \beta^{-1}\gamma ( M_n + 1)\lambda |\xi_{n+1}|^2,
    \end{aligned}
\end{equation*}
ince \[|K_n|^2\leq 2 |f_n|^2 +2|g_n|^2,\] then,
\begin{equation*}
    \E_n |K_n|^2\leq 2 C_3d\beta^{-1} \lambda \gamma M_n + \lambda \gamma C_4,
\end{equation*}
where $C_4=6(C'_0+\frac{d}{\beta})$ and
\begin{equation*}
    \begin{aligned}
\E_n |K_n|^{2q}&\leq 2^{2q-1}\E_n |f_n|^{2q} +2^{2q-1} \E_n |g_n|^{2q}\\&\leq 2^{2q-1} C_3^q \beta^{-q} \lambda ^q \gamma^q |M_n|^q \E |\xi_{n+1}|^{2q} +\lambda^{2q} \gamma^{2q} C_5 \E |\xi_{n+1}|^{4q} \\&\leq \lambda \gamma C_6 (|M_n|^q+1),
\end{aligned}
\end{equation*}
for $C_6= C_3^q 2^{2q}(\frac{d}{\beta})^q$. Substituting the moments of $K_n$ into \eqref{eq-base} one obtains
\[\begin{aligned}
\E_n |M_{n+1}|^{2q}&\leq(1- \frac\lambda \gamma \frac{m}{32})|M_n|^{2q} + \left(2q \lambda \gamma( C_2+C_3d) |M_n|^{2q-1}-\frac\lambda \gamma \frac{m}{32}|M_n|^{2q}\right)\\&+\left(q(2q-1)2^{2q-2} \lambda \gamma C_4 |M_n|^{2q-2}-\frac\lambda \gamma \frac{m}{32}|M_n|^{2q}\right)
\\&+\left(q(2q-1)2^{2q-3}\lambda \gamma C_6 |M_n|^q -\frac{\lambda} {\gamma} \frac{m}{3}|M_n|^{2q}\right)
+\lambda \gamma C_6q(2q-1)2^{2q-3}.
\end{aligned}\]
As a result there exists $N$ independent of $\gamma$ given by \begin{equation}
\begin{aligned}
  N&=\left(\frac{32}{m}\right)^{2q-1}(2q(C_2+C_3d))^{2q}+((\frac{32}{m})^{q-1}q(2q-1)2^{2q-2}C_4)^{q}+(\frac{32}{m})^q(q(2q-1)2^{2q-3}C_6)^{q+1} \\&+C_6q(2q-1)2^{2q-3},
\end{aligned}
\end{equation}
such that
\begin{equation*}
    \E_n |M_{n+1}|^{2q}\leq (1-\frac{\lambda m}{32 }\gamma^{-1}) |M_n|^{2q} + N \lambda \gamma^{4q-1},
\end{equation*}
which implies
\[\sup_n \E  |M_n|^{2q}\leq \E |M_0|^{2q} + \frac{32 N}{m} \gamma^{4q}.\]
Applying \eqref{eq-inequal} one concludes that
\[\E |\bar{\theta}^\lambda_n|^{4q}\leq 8^{2q} \frac{1}{\gamma^{4q}}(\frac{1}{1-2r})^{2q} \E |M_n|^{2q}\leq \bar{C}_{4q},\]
so that \[\E|\bar{\theta}^\lambda_n|^{2q}\leq \sqrt{\bar{C}_{4q}}.\]
\end{proof}
\begin{proof}[\textbf{Proof of Lemma} \ref{one-step}]
Since for every $t\geq 0$
\[V^\lambda_t =V^\lambda_{\lfloor t \rfloor}-\lambda \gamma \int_{\lfloor t\rfloor} ^t V^\lambda_{\lfloor s\rfloor}ds-\lambda \int_{\lfloor t\rfloor}^t h_{tam,\gamma}(\theta^\lambda_{\lfloor s \rfloor})ds+ \sqrt{2\gamma\lambda \beta^{-1}}(B^\lambda_t-B^\lambda_{\lfloor t \rfloor}).\]
one has that
\begin{equation}\label{eq-onestep}
    \begin{aligned}
     \E |V_t^\lambda-V_{\lfloor t \rfloor}|^2&\leq 8\lambda^2 \gamma^2 \E |\int_{\lfloor t \rfloor}^t \Vfls ds| ^2 +8\lambda^2 \E |\int_{\lfloor t\rfloor}^t h_{tam,\gamma}(\theta^\lambda_{\lfloor s \rfloor})ds|^2 + 16\frac{\lambda \gamma}{\beta} \E |B_{\lfloor t \rfloor}-B_t|^2\\&\leq 8\lambda^2 \gamma^2 \int_{\lfloor t \rfloor}^t \E |\Vfls|^2 ds +8\lambda^2 \int_{\lfloor t \rfloor}^t \E \big|h_{tam,\gamma}(\tfls)\big | ^2 ds +16\frac{\lambda \gamma}{\beta} d\\&\leq
     8\lambda^2 \gamma^2 \int_{\lfloor t \rfloor}^t \E |\Vfls|^2 ds +8\lambda^2 (A \int_{\lfloor t \rfloor}^t \E |\tfls|^2ds +4\gamma ) +16\frac{\lambda\gamma d}{\beta}.
    \end{aligned}
\end{equation}
Then, noting that the law of the continuous interpolation agrees with the algorithm at grid points, the moment bounds of the previous section yield \[\E |V^\lambda_{\lfloor s \rfloor}|^2\leq \bar{B}_2,\] and \[\E |\tfls|^2\leq \bar{C}_2,\] so
\eqref{eq-onestep} yields
\[ \E |V_t^\lambda-V^\lambda_{\lfloor t \rfloor}|^2\leq C_{1,v}\lambda \gamma.\]
In addition, it easy to see that
\[\theta^\lambda_t =\theta^\lambda _{\lfloor t \rfloor}+ \lambda \int_{\lfloor t\rfloor}^tV^\lambda_{\lfloor s \rfloor}ds, \] so that
\[\E |\theta^\lambda_t-\theta^\lambda_{\lfloor t \rfloor}|^2\leq  \lambda  \bar{B}_2.\]
\end{proof}

\section{Proofs of convergence results}
\begin{proof}[\textbf{Proof of Theorem} \ref{Aux-alg distance}]
Firstly one calculates
\begin{equation}
    \begin{aligned}
   \frac{d}{ds} \E |V_t^\lambda-Z_t^{\lambda,n}|^2&=-2 \lambda \gamma \E   \langle V_t^\lambda-Z_t^{\lambda,n},\Vflt-Z_t^{\lambda,n}\rangle  \\&-2\lambda  \E \langle V_t^\lambda-Z_t^{\lambda,n},h_{tam,\gamma} (\tflt)-h_{MY,\epsilon}(\zeta_t^{\lambda,n})\rangle
    \\&=-2 \lambda \gamma   \E\langle V_t^\lambda-Z_t^{\lambda,n},\Vflt-V_t^\lambda\rangle  -2\lambda  \gamma  \E |V_t^\lambda-Z_t^{\lambda,n}|^2 \\&-2\lambda  \E \langle V_t^\lambda-Z_t^{\lambda,n},h_{tam,\gamma} (\tflt)-h_{MY,\epsilon}(\zeta_t^{\lambda,n})\rangle
    \\&\leq\lambda\frac{\gamma}{2}   \E |V^\lambda_t-Z^{\lambda,n}_t|^2 +  2\lambda\gamma   \E |\Vflt-V_t^\lambda|^2 -2\lambda\gamma   \E |V_t^\lambda-Z_t^{\lambda,n}|^2 \\&+\lambda\frac{\gamma}{2}   \E |V_t^\lambda-Z_t^{\lambda,n}|^2 +2\lambda \gamma^{-1}  \E |h_{tam,\gamma}(\tflt)-h_{MY,\epsilon}(\zeta_t^{\lambda,n})|^2  \\&\leq -\lambda\gamma   \E |V_t^\lambda-Z_t^{\lambda,n}|^2 +2\lambda\gamma   \E |\Vflt-V_t^\lambda|^2 \\&+2\lambda \gamma^{-1}  \E |h_{tam,\gamma}(\tflt)-h_{MY,\epsilon}(\zeta_t^{\lambda,n})|^2
    \\&\leq 2\lambda \gamma^{-1}  \E |h_{tam,\gamma}(\tflt)-h_{MY,\epsilon}(\zeta_t^{\lambda,n})|^2 +2\lambda\gamma^2C_{1,v}-\lambda\gamma   \E |V_t^\lambda-Z_t^{\lambda,n}|^2\\&\leq
  \lambda( \underbrace{ 16 \gamma^{-1}  \E |h_{tam,\gamma}(\tflt)-h(\tflt)|^2}_{A_t}\\&
    +\underbrace{16 \gamma^{-1}   \E |h(\tflt)-h_{MY,\epsilon}(\tflt)|^2}_{B'_t}
    \\&\underbrace{ +16 \gamma^{-1}   \E |h_{MY,\epsilon}(\tflt)-h_{MY,\epsilon}(\theta^\lambda_t)|^2}_{C_t}
    \\&+\underbrace{16 \gamma^{-1}   \E |h_{MY,\epsilon}(\theta^\lambda_t)-h_{MY,\epsilon}(\zeta_t^{\lambda,n})|^2}_{D_t}
    \\&+2 \lambda\gamma^2 C_{1,v})-\lambda\gamma \E  |V_t^\lambda-Z_t^{\lambda,n}|^2
    \\&\leq \lambda\left( A_t+B'_t+C_t+D_t+2\lambda\gamma^2C_{1,v}\right)-\lambda\gamma   \E |V_t^\lambda-Z_t^{\lambda,n}|^2 .
    \end{aligned}
\end{equation}
Then using Lemma \ref{tamed distance} it is clear that
 \begin{equation}\label{eq-At}
   A_t=C_A\gamma^{-(2q+1)} .
\end{equation}
In addition, since $h_{MY,\epsilon}$ is $K\leq \frac{1}{\epsilon}$-Lipschitz and by Lemma \ref{one-step} it holds that \begin{equation}\label{eq-Ct}
    C_t\leq \gamma^{-1}K^2\lambda \bar{B}_2,
\end{equation}
and
\begin{equation}\label{eq-Dt}
    D_t\leq 16\lambda \gamma^{-1}K^2   \E |\theta^\lambda_t-\zeta_t^{\lambda,n}|^2\leq \frac{1}{4}\lambda   \gamma \E |\theta^\lambda_t-\zeta_t^{\lambda,n}|^2.
\end{equation}
Additionally, by Lemma \ref{gradient dist} one has that
 \[\begin{aligned}
|h(x)-h_{MY,\epsilon}(x)|^2&\leq  2^{4l+4}L^2(1+2|x|+\sqrt{R})^{4l+4} \epsilon^2,
\end{aligned}\]
and as a result, \begin{equation}\label{eq-Bt}
    B'_t\leq C_{B'} \gamma^{-1}\epsilon^2,
\end{equation}
where $C_{B'}=2^{8l+8}L^2 \left((1+\sqrt{R})^{4l+4} +\bar{C}_{4l+4}\right)$. Now let \begin{equation}\label{eq-slg}
    \slg=C_A\gamma^{-(2q+1)} +\lambda \gamma^{-1}K^2 \bar{B}_2+C_{B'}\gamma^{-1}\epsilon^2+C_{1,v}\lambda \gamma^2,
\end{equation}
so that bringing \eqref{eq-At}, \eqref{eq-Bt}, \eqref{eq-Ct}, \eqref{eq-Dt} together yields
\begin{equation}\label{eq-basic}
   \frac{d}{dt} \E |V_t^\lambda-Z_t^{\lambda,n}|^2\leq -\lambda\gamma  \E |V_t^\lambda-Z_t^{\lambda,n}|^2 +\lambda \slg+\frac{1}{4}\lambda \gamma   \E |\theta^\lambda_t-\zeta_t^{\lambda,n}|^2.
\end{equation}
Furthermore, by Lemma \ref{mom-continuous} \[\sup_{t\geq nT} \E|\zeta_t^{\lambda,n}|^2<\infty,\] and so by the bound
\[\E |\theta^\lambda_t-\zeta_t^{\lambda,n}|^2\leq 4 \E |\zeta_t^{\lambda,n}|^2 +4\E |\theta^\lambda_{\lfloor t\rfloor}|^2 +4\E |\theta_t^\lambda-\theta^\lambda_{\lfloor t \rfloor}|^2,\] one has,
\[\sup_{nT\leq t \leq {(n+1)T}} \E  |\theta^\lambda_t-\zeta_t^{\lambda,n}|^2<\infty.\]
Then one can define $J:=\sup_{nT\leq t \leq {(n+1)T}} \E  |\theta^\lambda_t-\zeta_t^{\lambda,n}|^2$
so that \eqref{eq-basic} becomes
\[\frac{d}{dt} \E |V_t^\lambda-Z_t^{\lambda,n}|^2\leq -\lambda\gamma  \E |V_t^\lambda-Z_t^{\lambda,n}|^2 +\lambda \slg+\lambda\gamma\frac{J}{4},\]
which implies
\[\left(e^{\lambda\gamma t}\E |V_t^\lambda-Z_t^{\lambda,n}|^2 -e^{\lambda\gamma t}(\frac{\slg}{\gamma}+\frac{J}{4 }\right)'\leq 0,\]
so by the fundamental theorem of calculus \[e^{\lambda\gamma t}\E |V_t^\lambda-Z_t^{\lambda,n}|^2 -e^{\lambda\gamma t}(\frac{\slg}{\gamma}+\frac{J}{4 })\leq -e^{\lambda \gamma nT}(\frac{\slg}{\gamma}+\frac{J}{4 }),\]
which leads to
\begin{equation}\label{eq-VZ bound}
\E |V_t^\lambda-Z_t^{\lambda,n}|^2\leq \frac{\slg}{\gamma}+\frac{J}{4 } \quad  t \in [nT,(n+1)T].
\end{equation}
On the other hand, one notices that for every $ t\in [nT,(n+1)T]$
\begin{equation}\label{eq-basic2}
    \begin{aligned}
 \E |\theta^\lambda_t-\zeta_t^{\lambda,n}|^2&\leq \lambda^2 \E \biggr (\int_{nT}^t |\Vfls-Z_s^{\lambda,n}|ds \biggr)^2\leq \lambda \int_{nT}^t \E |\Vfls-Z_s^{\lambda,n}|^2ds
 \\&\leq 2\lambda \int_{nT}^t \E |V_s^\lambda-Z_s^{\lambda,n}|^2ds +2 \lambda \int_{nT}^t \E |V_s^\lambda-\Vfls|^2 ds\\&\leq 2\lambda \int_{nT}^t \E |V_s^\lambda-Z_s^{\lambda,n}|^2 ds +2C_{1,v}\lambda \gamma \\&\leq \frac{2\slg}{\gamma}+\frac{J}{2} +2C_{1,v}\lambda \gamma.
 \end{aligned}
\end{equation}
so that taking supremum in the inequality yields
\[J\leq \frac{2\slg}{\gamma}+\frac{J}{2} +2C_{1,v}\lambda \gamma,\]
which implies that
\begin{equation}\label{eq-zth-bound}
\E |\theta^\lambda_t-\zeta_t^{\lambda,n}|^2\leq J\leq \frac{4\slg}{\gamma}+4 C_{1,v}\lambda \gamma \quad  t \in [nT,(n+1)T],
\end{equation}
so that substituting this back into \eqref{eq-VZ bound} yields
\begin{equation}\label{eq-Vz final}
    \E |V_t^\lambda-Z_t^{\lambda,n}|^2\leq \frac{5}{4} \frac{\slg}{\gamma}+ C_{1,v}\lambda.
\end{equation}
Bringing \eqref{eq-zth-bound} and \eqref{eq-Vz final} together, one concludes that
\[W_2^2\left(\mathcal{L}(V_t^\lambda,\theta^\lambda_t),\mathcal{L}(Z_t^{\lambda,n},\zeta_t^{\lambda,n})\right)\leq 6(C_{1,v}+1)\left(\frac{\slg}{\gamma}+\lambda\gamma\right).\]
\end{proof}

\begin{proof} [\textbf{Proof of Theorem} \ref{Contr}]
Using the results obtained in Theorem \ref{Contraction} and \ref{Aux-alg distance} one calculates for every $t\in [nT,(n+1)T]$ that
\[\begin{aligned}
&W_2(\mathcal{L}(\zeta^{\lambda,n}_t,Z^{\lambda,n}_t),\mathcal{L}(r^{\lambda}_t,R^\lambda_t))\leq \sum_{k=1}^n W_2(\mathcal{L}(\zeta^{\lambda,k}_t,Z^{\lambda,k}_t),\mathcal{L}(\zeta^{\lambda,k-1}_t,Z^{\lambda,k-1}_t)\\&=\sum_{k=1}^n
W_2( \mathcal{L}(\hat{\zeta}_t^{\lambda,kT,\theta^\lambda_{kT},V_{kT}},\hat{Z}_t^{\lambda,kT,\theta^\lambda_{kT},V_{kT}}),\mathcal{L}(\hat{\zeta}^{\lambda,kT,\zeta^{\lambda,k-1}_{kT},Z^{\lambda,k-1}_{kT}},\hat{Z}^{\lambda,kT,\zeta^{\lambda,k-1}_{kT},Z^{\lambda,k-1}_{kT}}))\\&\leq
\sum_{k=1}^n \sqrt{2}e^{-\frac{m}{\beta \gamma}\lambda (t-kT)}W_2(\mathcal{L}(\theta^\lambda_{kT},V^\lambda_{kT}),\mathcal{L}(\zeta^{\lambda,k-1}_{kT},Z^{\lambda,k-1}_{kT})\\&\leq
\sqrt{2}C\sum_{k=1}^n e^{-\frac{m}{\beta \gamma}(n-k)}   \left(\sqrt{\lambda}\sqrt{\gamma}+\sqrt{\frac{\slg}{\gamma}}\right)
\\&\leq \sqrt{2}C\left(\sqrt{\lambda}\sqrt{\gamma}+\sqrt{\frac{\slg}{\gamma}}\right)\frac{1}{1-e^{-\frac{m}{\beta \gamma}}}
\\&\leq \sqrt{2}C\left(\sqrt{\lambda}\sqrt{\gamma}+\sqrt{\frac{\slg}{\gamma}}\right) \frac{\beta}{m} \gamma e^{\frac{m}{\beta \gamma}}
\\&\leq C'\left(\sqrt{\lambda}\gamma^\frac{3}{2} +\sqrt{\gamma} \sqrt{\slg}\right),
\end{aligned}\]
\end{proof}

\begin{proof}[\textbf{Proof of Theorem} \ref{Final rate}]
By the triangle inequality of Wasserstein distance and by Corollaries \ref{W2 estimate}, \ref{contrinv}, Lemma \ref{tamed distance} and Theorem \ref{Contr}, for every $t\in [nT,(n+1)T]$ there holds that
\[\begin{aligned}
 W_2(\mathcal{L}(V^\lambda_t,{\theta^\lambda_t}),\pi_\beta)&\leq
 W_2(\mathcal{L}(V^\lambda_t,\theta^\lambda_t),\mathcal{L}(Z^{\lambda,n}_t,Z^{\lambda,n}_t)+W_2(\mathcal{L}(Z^{\lambda,n}_t,Z^{\lambda,n}_t),\mathcal{L}(R^\lambda_t,r^\lambda_t)\\&+ W_2(\mathcal{L}(R^\lambda_t,r^\lambda_t),\pi_\beta^\epsilon)+W_2(\pi_\beta^\epsilon,\pi_\beta)\\&\leq
 C\left( \sqrt{\frac{\slg}{\gamma}}+\sqrt{\lambda}\sqrt{\gamma}\right)+C'\left(\sqrt{\lambda}\gamma^\frac{3}{2} +\sqrt{\gamma} \sqrt{\slg}\right)\\&+\sqrt{2}e^{-\frac{m}{\beta \gamma}\lambda n}W_2(\mathcal{L}(\theta_0,V_0),\pi_\beta^\epsilon)+W_2(\pi_\beta,\pi_\beta^\epsilon)
 \\&\leq C\left( \sqrt{\frac{\slg}{\gamma}}+\sqrt{\lambda}\sqrt{\gamma}\right)+C'\left(\sqrt{\lambda}\gamma^\frac{3}{2} +\sqrt{\gamma} \sqrt{\slg}\right)\\&+\sqrt{2}e^{-\frac{m}{\beta \gamma}\lambda t}W_2(\mathcal{L}(\theta_0,V_0),\pi_\beta)+
 W_2(\pi_\beta,\pi_\beta^\epsilon)\\&\leq
 6\max\{C,C',\sqrt{2}+\sqrt{2c/m}\}\left(\sqrt{\lambda}\gamma^\frac{3}{2}+\sqrt{\gamma}\sqrt{\slg}\right)\\&+\sqrt{2}e^{-\frac{m}{\beta \gamma}\lambda t}W_2(\mathcal{L}(\theta_0,V_0),\pi_\beta)+\epsilon \\&\leq
 \dot{C} \left(\sqrt{\lambda}\gamma^\frac{3}{2}+\gamma^{-q}
 \right)+\sqrt{2}e^{-\frac{m}{\beta \gamma}\lambda t}W_2(\mathcal{L}(\theta_0,V_0),\pi_\beta) +\epsilon,
\end{aligned}\]
at which point the result follows from the fact that $\mathcal{L}\left(\bar{V}^\lambda_{\lfloor t\rfloor},\bar{\theta}^\lambda_{\lfloor t\rfloor}\right)=\mathcal{L}\left(V^\lambda_{\lfloor t\rfloor},\theta^\lambda_{\lfloor t\rfloor}\right)$.
\end{proof}
\section{Proofs of Section \ref{se-tKLMC2}}
\begin{proof}[Proof of Lemma \ref{remarkgridpoints}]
For simplicity we prove that $\mathcal{L}(\bar{Y}^\lambda_1,\bar{x}^\lambda_1)=\mathcal{L}(\Tilde{Y}_\lambda,\Tilde{X}_\lambda)$.\\
     Expanding $\Tilde{Y}_\lambda$ one notices that
     \begin{equation}\label{eq-tildey}
     \begin{aligned}
         \Tilde{Y}_\lambda &=e^{-\lambda \gamma} \Tilde{Y}_0-e^{-\lambda \gamma}\hg(\Tilde{x}_0)\int_0^\lambda e^{\gamma s}ds +\sqrt{2\gamma \beta^{-1}}e^{-\lambda \gamma}\int_0^\lambda e^{\gamma s} dW_s
         \\&=\psi_0(\lambda) \Tilde{Y}_0 -e^{-\lambda \gamma}\hg(\Tilde{x}_0) \frac{1}{\gamma}(e^{\lambda \gamma}-1)+\sqrt{2\gamma \beta^{-1}}e^{-\lambda \gamma}\int_0^\lambda e^{\gamma s} dW_s
         \\&=\psi_0(\lambda) \Tilde{Y}_0 -\hg(\Tilde{x}_0) \int_0^\lambda \psi_0(s) ds +\sqrt{2\gamma \beta^{-1}}e^{-\lambda \gamma}\int_0^\lambda e^{\gamma s} dW_s
         \\&=\psi_0(\lambda) \Tilde{Y}_0 -\psi_1(\lambda)\hg(\Tilde{x}_0)   +\sqrt{2\gamma \beta^{-1}}e^{-\lambda \gamma}\int_0^\lambda e^{\gamma s} dW_s
     \end{aligned}
     \end{equation}
     One notices that $e^{-\lambda \gamma}\int_0^\lambda e^{\gamma s} dWs $ follows a zero mean Gaussian distribution and
\[\begin{aligned}
\E (e^{-\lambda \gamma}\int_0^\lambda e^{\gamma s} dW_s) (e^{-\lambda \gamma}\int_0^\lambda e^{\gamma s} dW_s) ^T&=e^{-2\lambda \gamma } \int_0^\lambda e^{2\gamma s}ds I_d\\&=\frac{1}{2\gamma}e^{-2\lambda \gamma}(e^{2\lambda \gamma}-1) I_d\\&=\frac{1}{2\gamma}(1-e^{-2\lambda \gamma})  I_d\\&=(\int_0^\lambda e^{-2\gamma s}ds) I_d\\&=(\int_0^\lambda \psi_0^2(s) ds) I_d
\end{aligned}
     \]
     For the second process,
  \begin{equation}\label{eq-tildex}
      \begin{aligned}
          \Tilde{x}_\lambda&=\Tilde{x}_0+ \int_0^\lambda \psi_0(s) ds\Tilde{Y_0} -\hg(\Tilde{x}_0)\int_0^\lambda \int_0^s e^{-\gamma(s-u)} du ds +\sqrt{2\gamma \beta^{-1}} \int_0^\lambda e^{-\gamma t}\int_0^t e^{\gamma s} dW_s dt
          \\&=\Tilde{x}_0 +\psi_1(\lambda)\Tilde{Y_0}-\hg(\Tilde{x}_0)\frac{1}{\gamma}\int_0^\lambda e^{-\gamma s}(e^{\gamma s}-1)  ds+\sqrt{2\gamma \beta^{-1}} \int_0^\lambda e^{-\gamma t}\int_0^t e^{\gamma s} dW_s dt
          \\&=\Tilde{x}_0 +\psi_1(\lambda)\Tilde{Y_0}-\hg(\Tilde{x}_0)\int_0^\lambda \psi_1(s) ds+\sqrt{2\gamma \beta^{-1}} \int_0^\lambda e^{-\gamma t}\int_0^t e^{\gamma s} dW_s dt
          \\&=\Tilde{x}_0 +\psi_1(\lambda)\Tilde{Y_0}-\psi_2(\lambda)\hg(\Tilde{x}_0)+\sqrt{2\gamma \beta^{-1}} \int_0^\lambda e^{-\gamma t}M_t dt
      \end{aligned}
  \end{equation}
  where $M_t:= \int_0^t e^{\gamma s} dW_s$.
     Since  \[d (e^{-\gamma t}M_t)= -\gamma e^{-\gamma t} M_t dt +e^{-\gamma t} dM_t= -\gamma e^{-\gamma t} M_t dt + dW_t \] or equivalently
     \[e^{-\gamma t} M_t=-\gamma\int_0^\lambda e^{-\gamma s} M_s ds + W_t\]
    setting $t=\lambda$ and using the fact that $M_\lambda=\int_0^\lambda e^{\gamma s}dW_s$, one deduces that
     \begin{equation}
    \int_0^\lambda e^{-\gamma t} M_t dt =\frac{1}{\gamma} \int_0^\lambda (1-e^{\gamma (t-\lambda)}) dW_t
     \end{equation}
     which is a zero-mean Gaussian distribution with
     \[\E ( \int_0^\lambda e^{-\gamma t} M_t dt)( \int_0^\lambda e^{-\gamma t} M_t dt)^T=\frac{1}{\gamma^2}\int_0^\lambda (1-e^{-\gamma t})^2 I_d=(\int_0^\lambda \psi_1^2(t)dt)I_d.\]
     Finally,
     \[\begin{aligned}
         \E \left(e^{-\lambda \gamma} \int_0^\lambda e^{\gamma s} dWs\right)\left(\frac{1}{\gamma}\int_0 (1-e^{\gamma (s-\lambda)}) dW_s\right)^T&= \frac{1}{\gamma}\left(\int_0^\lambda e^{\gamma (s-\lambda) }(1-e^{\gamma (s-\lambda)}) ds\right ) I_d\\&=\int_0^\lambda \psi_0(s)\psi_1(s)ds \quad I_d
     \end{aligned}\]
     We have essentially proved that the terms $\int_0^\lambda e^{-\gamma t} M_t dt$ appearing in \eqref{eq-tildex} and
     $e^{-\lambda \gamma} \int_0^\lambda e^{\gamma s}dW_s$ in \eqref{eq-tildey} follow Gaussian distributions with the required cross-covariance matrix.\\
     The result immediately follows.
     \end{proof}
\begin{proof}[Proof of Lemma \ref{mom-KLMC2}]
The proof begins by writing \begin{equation}\label{eq-conn}
    \begin{aligned}
    \bar{Y}_{n+1}&=(1-\lambda \gamma)\yn -\lambda \hg(\xn)+E_n +\sqrt{2\gamma\beta^{-1}}\Xi_{n+1}
    \\\bar{x}_{n+1}&=\xn +\lambda \yn +e_n+\sqrt{2\gamma\beta^{-1}}\Xi'_{n+1}
    \end{aligned}
\end{equation}
where \begin{equation}
    \begin{aligned}
    E_n&=(\psi_0(\lambda)-1+\lambda\gamma)\yn +(\lambda-\psi_1(\lambda))\hg(\xn)\\
    e_n&=(\psi_1(\lambda)-\lambda)\yn-\psi_2(\lambda)\hg(\xn).
    \end{aligned}
\end{equation}
By elementary calculations using the definition of $\psi_i$ one obtains
\begin{equation}\label{eq-lemmaconn1}
    \max\{|\psi_0-1+\lambda\gamma|,|\lambda-\psi_1(\lambda)|,|\psi_2(\lambda)|\}\leq (1+\gamma^2)\lambda^2.
\end{equation}
Additionally, using the definition of the covariance matrix of the $2d-$ Gaussian random variable and using the inequality $\sum_{i=0}^k x^{k}/k! \leq e^x$ for $k=1,2$, $x\leq 0$, there holds
\begin{equation}\label{eq-lemmaconn2}
    \begin{aligned}
    C_{11}(\lambda)&:=\E |\Xi_{n+1}|^2\leq \lambda d
    \\C_{22}(\lambda)&:=\E |\Xi'_{n+1}|^2\leq \frac{d}{3}\lambda^3
    \\ C_{12}(\lambda)&:= \E \langle \Xi_{n+1},\Xi'_{n+1}\rangle \leq \frac{d}{2}\lambda^2.
    \end{aligned}
\end{equation}
Using the expression in \eqref{eq-conn} we shall use the same construction as in the proof of \ref{second mom}, and we shall use \eqref{eq-lemmaconn1} and \eqref{eq-lemmaconn2} to bound the remaining terms.
\begin{equation}
    \frac{\gamma^2}{4}\E |\bar{x}_{n+1}+\gamma^{-1}\bar{Y}_{n+1}|^2=\frac{\gamma^2}{4}\E |\xn+\gamma^{-1}\yn-\gamma^{-1}\lambda \hg(\xn)|^2+\delta_2(n)
\end{equation}
where \begin{equation}
\begin{aligned}
\delta_2(n)&=\frac{\gamma}{2\beta}\left(C_{11}(\lambda)/2+\gamma^2C_{22}(\lambda)+\gamma C_{12}(\lambda)\right)+\frac{1}{4}\E |\gamma e_n+E_n|^2\\&+\frac{1}{2}\E \langle \gamma \xn+\yn-\lambda \hg(\xn),\gamma e_n +E_n\rangle
\\&\leq C_1\lambda \gamma + \gamma^2 \E |e_n|^2 +\E |E_n|^2
+\lambda^2\gamma^4 \E |\gamma \xn +\yn -\lambda \hg(\xn)|^2
\\&+\frac{\gamma^2}{2 \lambda^2\gamma^4} \E (|e_n|^2+|E_n|^2)
\\&\leq C_1 \lambda \gamma +C_2 \lambda^2 \gamma^6 \E |\xn|^2 + C_3 \lambda^2 \gamma^4\E |\yn|^2+ C_4 \lambda^2 \E |\hg(\xn)|^2
\end{aligned}
\end{equation}
where the last steps where obtained by elementary inequalities and the use of \eqref{eq-lemmaconn1}.
Furthermore
\begin{equation}
    \frac{1}{4}\E |\bar{Y}_{n+1}|^2=\frac{1}{4}\E |(1-\lambda \gamma \bar{Y}_n-\lambda\hg(\xn)|^2+\delta_3(n)
\end{equation}
where  \begin{equation}
    \begin{aligned}
    \delta_3(n)&=\frac{1}{4}\E |E_n|^2 +\frac{1}{2}\E \langle (1-\lambda \gamma \yn-\lambda \hg(\xn),E_n\rangle+\frac{\gamma}{2\beta}C_{11}(\lambda)
    \\&\leq C_1 \lambda^2 \gamma^4( \E |\yn|^2+\lambda^2 |\hg(\xn)|^2) +
    C'_2 \lambda \gamma d\beta^{-1}.
    \end{aligned}
\end{equation}
Finally, \begin{equation}
    \begin{aligned}
    -\frac{\gamma^2r}{4}\E |\bar{x}_{n+1}|^2\leq -\frac{1}{4}\gamma^2 r \E |\xn +\lambda \yn|^2 +\delta_4(n)
    \end{aligned}
    \end{equation}
    where
    \begin{equation}
       \begin{aligned}
       \delta_4(n)&=-\frac{1}{2}\gamma^2 r \E \langle \xn+\lambda \yn,e_n\rangle+C_{22}(\lambda)\\&\leq C'_4 \lambda ^2 \gamma^4 \E |\xn+\lambda \yn|^2 + C'_4 \frac{1}{\lambda^2} \E |e_n|^2+\frac{d}{3}\lambda^3\\&\leq
       C'_5 \lambda^2 \gamma^4 \big(\E |\hg(\xn)|^2 +\E |\yn|^2 +\E |\xn|^2\big).
       \end{aligned}
    \end{equation}
As a result, one notices that
\begin{equation}
    \sum_{i=2}^4 \delta_i(n)\leq C'_6 \lambda^2 \gamma^6 (\E |\xn|^2 +\E |\yn|^2) +C_7' \lambda \gamma.
\end{equation}
Adapting the proof of \ref{second mom}, and defining
\begin{equation}
M_n=\frac{\gamma^2}{4} |\xn+\gamma^{-1}\bar{Y}_n|^2 +\frac{1}{4} |\bar{Y}_n|^2-\frac{r\gamma^2}{4} |\xn|^2
\end{equation}
as before, there holds
\[\begin{aligned}
\E [M_{n+1}|M_n]&\leq (1-\lambda r\gamma) M_n + \lambda \left(\lambda\gamma^2\frac{1}{4} +\frac{1}{2m}+\frac{1}{4}+\frac{r}{2}\gamma -\frac{\gamma}{2}\right) \E |\yn|^2
    \\& +\lambda \left( -\frac{m\gamma}{4}+r\gamma^3 +\frac{1}{2}\lambda m^2 +\frac{m^3}{16}\right)\E |\xn|^2 + \sum_{i=2}^4 \delta_i(n) +C_0 \lambda \gamma
    \\&\leq (1-\lambda r\gamma) M_n + \lambda \left(\lambda\gamma^2\frac{1}{4} +\frac{1}{2m}+\frac{1}{4}+\frac{r}{2}\gamma -\frac{\gamma}{2} +C'_6\lambda \gamma^6 \right) \E |\yn|^2
    \\& +\lambda \left( -\frac{m\gamma}{4}+r\gamma^3 +\frac{1}{2}\lambda m^2 +\frac{m^3}{16}+C'_6\lambda \gamma^6 \right)\E |\xn|^2 +C'_8 \lambda \gamma
\end{aligned}\]
If $\lambda <\gamma^{-5} \frac{1}{4C'_6}$ the terms that multiply $\E|\xn|^2$ and $\E |\yn|^2$ are negative so
\[\E M_{n+1}\leq (1-\lambda r \gamma)\E M_n +C'_8 \lambda \gamma\]
which yields the result.
\end{proof}
\begin{proof}[\textbf{Proof of Lemma \ref{HighmomKLMC2}}]
Applying once again the method used in the proof of Lemma \ref{high mom} one sees that
\begin{equation}
M_{n+1}\leq (1-r\lambda \gamma) M_n + C'_K K'_n
\end{equation}
where $C'_K$ is an absolute constant and $K'_n=K'_{1,n}+K'_{2,n}+K'_{3,n}+C'_8\lambda \gamma $ is given by \begin{equation}
    \begin{aligned}
    K'_{1,n}&= \sqrt{2\gamma\beta^{-1}}\langle \gamma \xn +\yn -\lambda \hg(\xn), \Xi_{n+1}+\Xi'_{n+1} \rangle \\&+\frac{8\gamma}{\beta}\left( |\Xi_{n+1}|^2+|\Xi'_{n+1}|^2\right)
    \\& +\sqrt{2\gamma \beta^{-1}}\langle E_n+\gamma e_n, \Xi_{n+1}+\Xi'_{n+1}\rangle
    \\&\leq C \sqrt{2\gamma \beta^{-1}} \sqrt{M_n}(|\Xi_{n+1}|+|\Xi'_{n+1}|)
    \\&+\frac{8\gamma}{\beta}\left( |\Xi_{n+1}|^2+|\Xi'_{n+1}|^2\right)
    \\& +\sqrt{2\gamma \beta^{-1}} \lambda^2 (1+\gamma^2)(\sqrt{M_n}+1)|\Xi_{n+1}+\Xi'_{n+1}|,
    \end{aligned}
\end{equation}
\begin{equation}
    \begin{aligned}
    K'_{2,n}&= \sqrt{2\gamma \beta^{-1}}\langle (1-\lambda \gamma)\yn -\lambda \hg(\xn),\Xi_{n+1}\rangle +2\gamma\beta^{-1}|\Xi_{n+1}|^2+\sqrt{2\gamma\beta^{-1}}\langle E_n,\xi_{n+1}\rangle \\&\leq C`_2\left(\sqrt{\gamma\beta^{-1}} (\sqrt{M_n}+1)|\Xi_{n+1}|+2\gamma\beta^{-1}|\Xi_{n+1}|^2+\lambda^2(1+\gamma^2)\frac{2\gamma}{\beta}\sqrt{M_n}|\Xi_{n+1}|\right)
    \end{aligned}
\end{equation}
and
\begin{equation}
    \begin{aligned}
    K'_3=-\frac{1}{2}\gamma^2 r \langle \xn +\lambda \yn ,\sqrt{2\gamma\beta^{-1}}\Xi'_{n+1}\rangle \leq C`_3 \sqrt{2\gamma \beta^{-1}}\sqrt{M_n}|\Xi'_{n+1}|.
    \end{aligned}
\end{equation}
Combining the expressions for $K'_1,K'_2,K'_3$ yields
\begin{equation}
    K_n\leq C_K \sqrt{2\beta^{-1}} \sqrt{\gamma} \sqrt{M_n}(|\Xi_{n+1}|+|\Xi'_{n+1}|) +\frac{\gamma}{\beta}(|\Xi_{n+1}|^2+|\Xi'_{n+1}|^2)+\lambda \gamma
\end{equation}
so using that $\lambda^2 \gamma^2\leq \lambda \gamma$ and the expressions for the expectations of the squares of $\Xi,\Xi'$ one obtains the crucial estimates
\begin{equation}
    \E_n K'_n=C_{K,1}\lambda\gamma
\end{equation}
\begin{equation}
    \E_n K_{n} ^2\leq C_{K,2}\lambda \gamma (M_n + 1)
\end{equation}
and
\[\E_n K_n ^{2p}\leq C_{K,p}\lambda^p \gamma^p (M_n^p +1)\]
where $C_{K,p}$ has at most $(\frac{d}{\beta})^p$ dependence on the dimension.

Using the same techniques as in the proof of the higher moments of KLMC1 one obtains
\[\begin{aligned}
\E_n |M_{n+1}|^{2p}&\leq  s |M_n|^{2p} +2q |M_n|^{2p-1} \E_n K_n\\&+p(2p-1)2^{2p-3}|M_n|^{2p-2}\E_n |K_n|^2+p(2p-1)2^{2p-3}\E_n |K_n|^{2p}
\\&\leq 2C_{K,1}s M_n^2p q\lambda \gamma M_n^{2q-1} +\lambda\gamma p(2p-1)2^{2p-3}M_n^{(2p-1)}+(C_{K,1}+C_{K,p})\lambda \gamma
\end{aligned}\]
Following again the same steps as in that proof one deduces that there exists $N'$ such that
\begin{equation}
    \E _n M_{n+1}^{2p}\leq (1-\lambda \gamma^{-1}m/32) M_n^{2p}+\lambda  N'\gamma^{4q-1}
\end{equation}
which leads to
\[\sup_n \E  |M_n|^{2q}\leq \E |M_0|^{2q} + \frac{32 N'}{m} \gamma^{4q}.\]
Since
\[\E |\xn|^{4q}\leq 8^{2q} \frac{1}{\gamma^{4q}}(\frac{1}{1-2r})^{2q} \E |M_n|^{2q}\leq \bar{C}_{4q},\]
one concludes that \[\E|\xn|^{2q}\leq \sqrt{\bar{C}_{4q}}.\]
\end{proof}
\begin{proof}[\textbf{ Proof of Lemma \ref{tKLMC2conv}} ]
One begins by calculating via Jensen's inequality
\begin{equation}\label{eq-obs}
    \E |\tilde{x}_t-\p_t|^2=\E |\int_{n\lambda}^t \Tilde{Y}_s-\q_s ds|^2 \leq (t-n\lambda) \E \int_{n\lambda}^t | \Tilde{Y}_s-\q_s|^2 ds\leq \lambda  \int_{n\lambda}^t \E | \Tilde{Y}_s-\q_s|^2 ds.
\end{equation}
\[\begin{aligned}
\E |\tilde{Y}_t-\q_t|^2&=\E \left | \int_{n\lambda}^t e^{-\gamma(t-s)}( \he(\p_s)-\hg(\tilde{x}_{n\lambda}))ds\right|^2
    \\&{\leq} (t-n\lambda) \int_{n\lambda}^t e^{-2\gamma(t-s)} \E |\he(\p_s)-\hg(\tilde{x}_{n\lambda})|^2 ds\\&\leq
    \lambda \int_{n\lambda}^t \E |\he(\p_s)-\hg(\tilde{x}_{n\lambda})|^2 ds\\&\leq
    8\lambda \int_{n\lambda}^t \E |\he(\p_s)-\he(\tilde{x}_{n\lambda})|^2ds
    \\&+  8\lambda \int_{n\lambda}^t \E |\he(\tilde{x}_{n\lambda})-h(\tilde{x}_{n\lambda})|^2 ds
    \\& + 8 \lambda \int_{n\lambda}^t \E |h(\tilde{x}_{n\lambda})-\hg(\tilde{x}_{n\lambda})|^2 ds.
\end{aligned}\]
For the first term, a further analysis yields
\begin{equation}\label{eq-term1}
\begin{aligned}
8\lambda \int_{n\lambda}^t \E &|\he(\p_s)-\he(\tilde{x}_{n\lambda})|^2ds \leq 8\lambda K^2 \int_{n\lambda} ^t \E |\p_s -\tilde{x}_{n\lambda}|^2ds\\&\leq
16\lambda K^2 \int_{n\lambda}^t \E|\tilde{x}_s -\tilde{x}_{n\lambda}|^2ds
+16\lambda K^2 \int_{n\lambda} ^t \E |{\tilde{x}_s-\p_s}|^2
\\&\leq  16\lambda K^2 \int_{n\lambda}^t \E |\int_{n\lambda}^s \tilde{Y}_u du|^2 ds + 16\lambda K^2 \int_{n\lambda}^t \E|\tilde{x}_s -\p_s|^2ds  \\&\leq
16 K^2 \lambda^4 \sup_{s\in [n\lambda,(n+1)\lambda} \E |\tilde{Y_s}|^2 +16\lambda K^2 \int_{n\lambda}^t \E|\tilde{x}_s -\p_s|^2ds  .
\end{aligned}
\end{equation}
For the second term,
there holds
\begin{equation}\label{eq-term2}
 8\lambda \int_{n\lambda}^t \E |\he(\tilde{x}_{n\lambda})-h(\tilde{x}_{n\lambda})|^2 ds\leq C_p \lambda^2 \epsilon^2
\end{equation}
For the third term,
\begin{equation}\label{eq-term3}
    8 \lambda \int_{n\lambda}^t \E |h(\tilde{x}_{n\lambda})-\hg(\tilde{x}_{n\lambda})|^2 ds\leq C_3 \lambda^2 \gamma^{-2q+1}.
\end{equation}
Inserting \eqref{eq-term1} \eqref{eq-term2} \eqref{eq-term3} into \eqref{eq-obs} yields
\[\begin{aligned}
\sup_{n\lambda \leq s \leq t} \E |\tilde{x}_s-\p_s|^2&\leq \lambda^2 R_{\lambda,\gamma}+ 16\lambda^2 K^2 \int_{n\lambda}^t \int_{n\lambda}^s \sup_{n\lambda\leq u\leq z} \E|\tilde{x}_u-\p_u|^2  dzds \\&\leq \lambda^2 R_{\lambda,\gamma} +16\lambda^3 K^2 \int_{n\lambda}^t \sup_{n\lambda \leq u \leq s} \E |\tilde{x}_s-\p_s|^2 ds
\end{aligned}\]
An application of Grownwall's inequality yields
\begin{equation}\label{eq-rate1}
    \E |\tilde{x}_t-\p_t|^2\leq \lambda^2 R_{\lambda\gamma}e^{16\lambda^3 K^2 (t-n\lambda)}.
\end{equation}
Inserting this into \eqref{eq-term1} and combining all together yields
\[
\E |\tilde{Y}_t-\q_t|^2\leq (1+16K^2\lambda^4)R_{\lambda,\gamma}\leq 2 R_{\lambda,\gamma}\] where $R_{\lambda,\gamma}$ is given by
\begin{equation}\label{eq-Rlg}
    R_{\lambda,\gamma}=\lambda^4 16K^2\sup_{[n\lambda\leq s\leq (n+1)\lambda]}\E |\tilde{Y}_s|^2+ C_3 \lambda^2 \gamma^{-2q+1}  + C_p\lambda^2 \epsilon^2.
\end{equation}
\end{proof}
\begin{proof}[\textbf{Proof of Lemma \ref{Lemmacontra}}]
Similarly as in the proof of \ref{Final rate}, one calculates via Lemmas \ref{tKLMC2conv}, \ref{Lemmacontra}
\[\begin{aligned}
&W_2(\mathcal{L}(p^{\lambda,n},Q^{\lambda,n}),\mathcal{L}(X_t,Y_t))\leq \sum_{k=1}^n W_2(\mathcal{L}({p^{\lambda,k}},Q^{\lambda,k}),\mathcal{L}(p^{\lambda,k-1},Q^{\lambda,k-1}))
\\&\leq \sum_{k=1}^n W_2\left(\mathcal{L}(\hat{p}^{k\lambda,(\tilde{x}_{k\lambda},\tilde{Y}_{k \lambda})},\hat{Q}^{k\lambda,(\tilde{x}_{k\lambda},\tilde{Y}_{k \lambda})}),\mathcal{L}(\hat{p}^{k\lambda,(p^{\lambda,k-1}_{k\lambda},Q^{\lambda,k-1}_{k \lambda})},\hat{Q}^{k\lambda,(p^{\lambda,k-1}_{k\lambda},Q^{\lambda,k-1}_{k \lambda})})\right)
\\&\leq \sum_{k=1}^n \sqrt{2}e^{-\frac{m}{\beta \gamma}(t-k\lambda)} W_2(\mathcal{L}(\tilde{x}_{k\lambda},\tilde{Q}_{k \lambda}),\mathcal{L}(p^{\lambda,k-1}_{k\lambda},Q^{\lambda,k-1}_{k\lambda}))
\\&\leq \sum_{k=1}^n \sqrt{2}e^{-\frac{m}{\beta \gamma}\lambda(n-k)} W_2(\mathcal{L}(\tilde{x}_{k\lambda},\tilde{Q}_{k \lambda}),\mathcal{L}(p^{\lambda,k-1}_{k\lambda},Q^{\lambda,k-1}_{k\lambda}))\\&\leq
\sqrt{2}C \sqrt{R_{\lambda,\gamma}}\sum_{k=1}^n \sqrt{2}e^{-\frac{m}{\beta \gamma}\lambda(n-k)}\\&\leq
\sqrt{2}C \sqrt{R_{\lambda,\gamma}} \frac{1}{1-e^{-\frac{m}{\beta \gamma}\lambda}}
\\&\leq \sqrt{2}C\sqrt{R_{\lambda,\gamma}}\frac{\beta \gamma}{m\lambda }e^{\frac{m}{\beta \gamma}\lambda}
\\&\leq C' \frac{\gamma}{\lambda} \sqrt{R_{\lambda,\gamma}}.
\end{aligned}\]
\end{proof}
\section{Proofs of section \ref{se-opt}}
\begin{proof}[\textbf{Proof of Lemma} \ref{opt-lemma2}]
By strong convexity there holds that for every $x\in E$
\begin{equation}\label{eq-stc}
u(x)-u(y)\leq \langle h(x),x-y\rangle \leq |h(x)||x-y|\leq L(1+2R_0)^{l+1}|x-y|.\end{equation}
Furthermore, since $\E u(x_n)=\E u(X_n\mathds{1}_E(X_n))$, setting $x:=X_n\mathds{1}_E(X_n)$ and $y:=X\mathds{1}_E(X_n)$ in \eqref{eq-stc} yields
\[\E u(x_n)-u(X\mathds{1}_E(X_n)\leq \E L(1+2R_0)^{l+1}|X_n\mathds{1}_E(X_n)-X\mathds{1}_E(X_n)|\leq L(1+2R_0)^{l+1}\E |X_n-X|, \] which by Cauchy-Swartz inequality implies
\[\E u(x_n)-u(X\mathds{1}_E(X_n)\leq L(1+2R_0)^{l+1}\sqrt{\E |X_n-X|^2}=L(1+2R_0)^{l+1}W_2(\mathcal{L}(\bar{\theta}_n^\lambda,\mu_\beta).\]
\end{proof}
\begin{proof}[\textbf{Proof of Lemma} \ref{opt-lemma2}]
We begin by calculating
\begin{equation}\label{eq-step1}
\begin{aligned}
 \E [u(X\mathds{1}_E(X_n))]-u(x^*)&\leq
 \E [u(X)\mathds{1}_E(X_n)]-u(x^*) \\&+u(0)P(X_n\in E^c)\\&=
 \E [u(X)\mathds{1}_E(X_n)\mathds{1}_{E'}(X)] + \E[ u(X)\mathds{1}_E(X_n)\mathds{1}_{E'^c}(X)]-u(x^*) \\&+u(0)P(X_n\in E^c)
 \\&\leq \E [u(X)\mathds{1}_{E'}(X)]-u(x^*)+ \sqrt{\E u^2(X)}\sqrt{P(\{X_n\in E\}\cap \{X\in E'^c\})}\\&+u(0)P(X_n\in E^c).
\end{aligned}
\end{equation}
In addition, one notices that $u$ is $M_0=L(1+2(R_0+1))^l$ smooth on $E'$ so
  \begin{equation}\label{eq-step2}
      \E [u(X)\mathds{1}_{E'}(X)]-u(x^*)\leq \int_{E'}u(x)-u(x^*)d\mu_\beta(x)\leq \frac{M_0}{2}\int_{E'}|x-x^*|^2d\mu_\beta(x)\leq \frac{M_0}{2}\E |X-x^*|^2.
  \end{equation}
Furthermore, since $u$ is convex
\[\begin{aligned}
 u(x)-u(0)&\leq \langle h(x)-h(x^*),x\rangle\\&\leq L(1+|x|+|x^*|)|x-x^*||x|\leq L(1+|x|+\frac{2u(0)}{m})^{l+2}\\&\leq 2^{l+2}L(1+\frac{2u(0)}{m})^{l+2} + 2^{l+2}L|x|^{l+2} ,
\end{aligned}
\]
and in addition \begin{equation}\label{eq-momu2}
    \E u^2(X)\leq 2u^2(0) + 2^{2l+8}L^2(1+\frac{2u(0)}{m})^{2l+4}+2^{2l+8} E|X|_{2l+4} .
\end{equation}
Then since
\[\begin{aligned}
 P(\{X_n\in E\}\cap \{X\in E'^c\})\leq P(|X_n-X|>1)\leq \E |X_n-X|^2,
\end{aligned}\]
one deduces that
\begin{equation}\label{eq-step3}
    \begin{aligned}
     ((E|X|_{2l+4})^{1/2}+ C')\sqrt{P(\{X_n\in E\}\cap \{X\in E'^c\})}\leq C''(||X||_{2l+4}^{l+2}+1)W_2(\bar{\theta}_n^\lambda,\mu_\beta).
    \end{aligned}
\end{equation}
Finally one has the bound
\begin{equation}\label{eq-step4}
    \begin{aligned}
     P(X_n\in E^c)\leq P(|X_n-x^*|>1)\leq \E |X_n-x^*|^2\leq 2\E |X-x^*|^2 +W_2(\bar{\theta}_n^\lambda,\mu_\beta),
    \end{aligned}
\end{equation}
so that inserting \eqref{eq-step2},\eqref{eq-step3},\eqref{eq-step4} into \eqref{eq-step1} completes the proof.

\end{proof}

\begin{proof}[\textbf{Proof of Lemma} \ref{opt-lemma3}]
Consider the Langevin SDE given as
\[d\bar{L}_t= -h(\bar{L}_t)dt +\sqrt{\frac{2}{\beta}} dB_t ,\]
with initial condition $L_0=x^*$. By Ito's formula, one obtains
\[\begin{aligned}
   d |\bar{L}_t-x^*|^2&=-2   \langle h(\bar{L}_t),\bar{L}_t-x^*\rangle dt + \frac{2d}{\beta} dt + dG_t \\&= -2   \langle h(\bar{L}_t)-h(x^*),\bar{L}_t-x^*\rangle dt +\frac{2d}{\beta} dt + dG_t\\&\leq -m  |\bar{L}_t-x^*|^2 dt +\frac{2d}{\beta} dt +dG_t.
\end{aligned}\]
where $dG_t$ is a martingale.
Taking expectations, and since by \eqref{eq-oversecondm} one has that $\sup_{t\geq0}\E|\bar{L}_t|^2<\infty$, we are able to interchange the differentiation with expectation so that
\[\frac{d}{dt} \E |\bar{L}_t-x^*|\leq -m \E |\bar{L}_t-x^*|^2 dt + \frac{2d}{\beta},\]
Then applying Grownwall's inequality yields
\[\E |\bar{L}_t-x^*|^2\leq \frac{2d}{\beta m}, \quad  t>0.\]
Due to the fact that $\bar{L}_t$ converges in the $W_2$ distance to the invariant measure, one obtains $L_1$ convergence and the bound $\E |\bar{L}_t|^2\rightarrow  \E |X|^2$, so since there also holds \[|\E |\bar{L}_t-x^*|^2-\E |X-x^*|^2|= \left| \E |\bar{L}_t|^2 -\E |X|^2 \right| + 2|x|^*\E |\bar{L}_t-X|,\]
there follows that
\[\E |X-x^*|^2=\lim_{t\rightarrow\infty} \E |\bar{L}_t-x^*|^2\leq \frac{2d}{m\beta}.\]
\end{proof}
\begin{proof}[\textbf{Proof of Corollary} \ref{opt-solution}]
Combining Lemmas \ref{opt-lemma1},\ref{opt-lemma2},\ref{opt-lemma3} and using the fact that \[W_2(\mathcal{L}(\bar{\theta}^\lambda_n),\mu_\beta)\leq W_2(\mathcal{L}(\bar{\theta}^\lambda_n,\bar{V}^\lambda_n),\pi_\beta),\] the result follows immediately.
\end{proof}
\newpage
\section{Table of constants}

\begin{table}[h]
     \renewcommand{\arraystretch}{2}
    \centering
    \caption{Basic constants and dependency on key parameters}
    \begin{tabular}{@{}lllll@{}}
    \toprule
     \multicolumn{1}{c}{Constant} &  \multicolumn{3}{c}{Key parameters}   \\ \hline
     \phantom{Constant} &$\frac{d}{\beta}$ & $m$& $\E |M_0|^p$\\
     $\tilde{C}_p$ & $(\frac{d}{\beta})^q$ & $m^{-q}$ & $\sqrt{\E |M_0|^{2p}}/\gamma^{2q}$\\
     $\tilde{B}_2$ & $\frac{d}{\beta}$& $(2m+1)^2$& $\E |V_0|^2 $
     \\$C_{\mu_\beta,p}$ &$(\frac{d}{\beta})^\frac{p}{2}$& $m^{-p/2}$ &-
     \end{tabular}
\end{table}
\begin{table}[h]
\renewcommand{\arraystretch}{2}
\centering
    \caption{Derived constants and their dependency to basic constants}
    \label{tab:derived constants}
    \begin{tabular}{@{}lllll@{}}
    \toprule
     \multicolumn{1}{c}{Constant} &  \multicolumn{4}{c}{Key parameters}   \\     \phantom{Constant}& $\beta$ & $\tilde{C}_p$ &$\tilde{B}_2$& $C_{\mu_\beta,p}$   \\ $C_{1,v}$ & - & $\tilde{C}_2$& - & -
     \\ $C_A$ &-& $\tilde{C}_{2(l+2)+2q(2l+2)}$ &-
     \\ $c$ &-& -& -& $\sqrt{C_{\mu_\beta,4l+4}}^{4l+4}$
     \\ $C$ &-&$\sqrt{\tilde{C}_{2(l+2)+2q(2l+2)}}$& $\sqrt{\tilde{B}_2}$ &-
     \\ $\dot{C}$ & $\beta$ & $\sqrt{\tilde{C}_{2(l+2)+2q(2l+2)}}$& $\sqrt{\tilde{B}_2}$ &$\sqrt{C_{\mu_\beta,4l+4}}^{4l+4}$
     \end{tabular}
     \end{table}
     \newpage
\bibliography{Convex_Hamiltonian}
\end{document}